\documentclass[11pt]{article}
\usepackage[utf8]{inputenc}
\usepackage{makeidx,amssymb,amscd}
\usepackage{amsthm}
\usepackage{amsmath}
\usepackage{amsfonts}
\usepackage[toc,page]{appendix}
\usepackage{graphicx}
\usepackage{mathrsfs}
\usepackage[colorlinks]{hyperref}
\newcommand{\R}{\mathbb{R}}

\newtheorem{theorem}{Theorem}
\newtheorem{lemma}{Lemma}
\numberwithin{equation}{section}

\newtheorem{corollary}{Corollary}
\newtheorem{proposition}{Proposition}
\newtheorem{definition}{Definition}

\newtheorem{proof of thm}{Proof of Theorem 1}

\date{\empty}
\begin{document}
\title{Local exponential stabilization for a class of Korteweg-de Vries equations by means of time-varying feedback laws}
\author{Jean-Michel Coron\thanks{Universit\'{e} Pierre et Marie Curie-Paris 6, UMR 7598 Laboratoire Jacques-Louis Lions, 75005 Paris, France. E-mail: \texttt{coron@ann.jussieu.fr}. JMC was supported by ERC advanced grant 266907 (CPDENL) of the 7th Research Framework Programme (FP7).},\;\; Ivonne Rivas\thanks{Universit\'{e} Pierre et Marie Curie-Paris 6, UMR 7598 Laboratoire Jacques-Louis Lions, 75005 Paris, France. E-mail: \texttt{rivas@ann.jussieu.fr}. IR was supported by ERC advanced grant 266907 (CPDENL) of the 7th Research Framework Programme (FP7).},\;\; Shengquan Xiang\thanks{Universit\'{e} Pierre et Marie Curie-Paris 6, UMR 7598 Laboratoire Jacques-Louis Lions, 75005 Paris, France. E-mail: \texttt{shengquan.xiang@ens.fr.}}}
\maketitle
\begin{abstract}
We study the exponential stabilization problem for a nonlinear Korteweg-de Vries equation on bounded interval in cases where the linearized control system is not controllable. The system has Dirichlet boundary conditions at the end-points of the interval, a Neumann nonhomogeneous boundary condition at the right end-point which is the control. We build a class of time-varying feedback laws for  which the solutions of the closed-loop systems with small initial data decay exponentially to $0$. We present also results on the well-posedness of the closed-loop systems for general time-varying feedback laws.
\end{abstract}

\section{Introduction}
\label{sec-introduction}
Let $L\in(0,+\infty)$. We consider the stabilization of the following controlled Korteweg-de Vries (KdV) system
\begin{gather}\label{kdv}
\begin{cases}
y_t+y_{xxx}+y_x+y y_x=0 \;\;\;\;\;&\textrm{for}\;\;  (t,x) \in (s,+\infty)\times (0,L), \\
y(t,0)= y(t,L)=0   &\textrm{for}\;\; t \in (s, + \infty),\\
y_x(t,L)=u(t)    &\textrm{for}\;\; t \in (s,+\infty),\\
\end{cases}
\end{gather}
where $s \in \mathbb{R}$ and where,  at time $t \in [ s, +\infty )$, the state is $y(t, \cdot) \in L^2(0,L)$  and the control is $u(t) \in \mathbb{R}$.

Boussinesq in \cite{1877-Boussinesq}, and Korteweg and de Vries in \cite{kdv} introduced the KdV equations for describing the propagation of small amplitude long water waves. For better understanding of KdV, one can see Whitham's book \cite{whitham}, in which different mathematical models of water waves are deduced. These equations have turned out to be good models not only for water waves but also to describe  other physical phenomena. For mathematical studies  on these equations, let us mention the following \cite{bona, constantin, craig, temam} and the references therein as well as the discovery of solitons   and the inverse scattering method \cite{cohen, gardner} to solve these equations. We also refer here to  \cite{bsz, bsz2, coron04, ruz, zhang} for well-posedness results of initial-boundary-value problems of our KdV equation \eqref{kdv} or for other equations which are similar to \eqref{kdv}. Finally, let us refer to \cite{cerpatu, rz} for reviews on recent progresses on the control of various KdV equations.

The controllability research on \eqref{kdv} began in 1997 when Lionel Rosier  showed in \cite{rosier97} that the  linearized KdV control system  (around $0$ in $L^2(0,L)$)
\begin{gather}\label{Lkdv}
\begin{cases}
y_t+y_{xxx}+y_x=0 \;\;\;\;\;&\textrm{in}\;\;  (0,T)\times (0,L), \\
y(t,0)= y(t,L)=0   &\textrm{on}\;\; (0,L),\\
y_x(t,L)=u(t)    &\textrm{on}\;\; (0,T),\\
\end{cases}
\end{gather}
is controllable if and only if $L \notin \mathcal{N}$, where $\mathcal{N}$ is called the set of critical lengths
 and is defined by
\begin{equation}\label{N}
\mathcal{N} := \Big \{    2 \pi \sqrt{\frac{l^2 + lk + k^2}{3}} ;\,  l, k \in \mathbb{N}^{*}   \Big \}.
\end{equation}
 From this controllability result Lionel Rosier, in the same article, deduced that the nonlinear KdV equations \eqref{kdv} are locally controllable (around $0$ in $L^2(0,L)$) if $L \notin \mathcal{N}$. His work also shows that {\color{black} the} $L^2(0,L)$ space can be decomposed as  $H \oplus M$, where $M$ is the ``uncontrollable'' part for the linearized KdV  control systems \eqref{Lkdv}, and $H$ is the ``controllable'' part.  Moreover, $M$ is of finite dimension, a dimension which is strongly depending on some number theory property of the length $L$. More precisely, the dimension of $M$ is  the number of different pairs of positive integers $(l_j,k_j)$ satisfying
\begin{equation}\label{defL}
 L = 2 \pi \sqrt{\frac{l_j^2 + l_jk_j + k_j^2}{3}}.
\end{equation}
 For each such pair of  $(l_j, k_j)$ with $l_j \geqslant k_j$, we can find two nonzero real valued functions $\varphi^j_1$ and $\varphi^j_2$ such that $\varphi^j:=\varphi^j_1 +  i\varphi^j_2$ is a solution of
\begin{gather}\label{varj}
\begin{cases}
-i\omega(l_j, k_j) \varphi^j + \varphi^{j'} +  \varphi^{j'''} = 0, \\
\varphi^j (0) = \varphi^j (L) =0,\\
\varphi^{j'} (0) = \varphi^{j'} (L) =0,
\end{cases}
\end{gather}
where $\varphi^j_1, \varphi^j_2 \in C^{\infty}([0,L])$ and $\omega(l_j, k_j)$ is defined by
\begin{equation}
\label{eq-omega}
  \omega (l_j, k_j) := \frac{(2l_j + k_j)(l_j - k_j)(2k_j + l_j)}{3\sqrt{3}(l_j^2 + l_j k_j + k_j^2)^{3/2}}.
\end{equation}
When $l_j > k_j$, the functions  $\varphi^j_1, \varphi^j_2$ are linearly independent, but when $l_j = k_j$ then $\omega(l_j, k_j)= 0 $ and  $\varphi^j_1, \varphi^j_2$ are linearly dependent. It is also proved in  \cite{rosier97}  that
\begin{equation}\label{M=spanphi}
  M = \textrm{Span} \{ \varphi^1_1, \varphi^1_2, ..., \varphi^n_1, \varphi^n_2 \}.
\end{equation}
Multiplying \eqref{Lkdv} by $\varphi^j $, integrating on $(0, L)$, performing integrations by parts and combining with \eqref{varj}, we get
 \[ \frac{d}{dt} \left( \int_0^L y(t,x) \varphi^j (x) dx \right) = i\omega(l_j, k_j)  \int_0^L y(t,x) \varphi^j (x) dx,  \]
which shows that $M$ is included in the  ``uncontrollable'' part of \eqref{Lkdv}. Let us point out that there exists at most {\color{black} one} pair of $(l_j, k_j)$ such that $l_j = k_j$. Hence we can classify $L \in \mathbb{R}^{+}$ in  5 different cases and therefore divide $\mathbb{R}^{+}$ into five disjoint subsets of $(0,+\infty)$, which are defined as follows:
\begin{enumerate}
\item  $\mathcal{C} := \mathbb{R}^{+}\setminus{\mathcal{N}}$. Then $M = \{ 0 \}$.
\item $\mathcal{N}_1 := \big{\{} L \in \mathcal{N};$ there exists one and only one ordered pair $(l_j, k_j)$ satisfying  \eqref{defL} and one has $l_j = k_j\big{\}}$. Then  the dimension of $M$ is 1.
\item $\mathcal{N}_2 := \big{\{} L \in \mathcal{N};$ there exists one and only one ordered pair  $(l_j, k_j)$ satisfying  \eqref{defL} and one has $l_j > k_j\big{\}}$. Then  the dimension of $M$ is 2.
\item $\mathcal{N}_3 := \big{\{} L \in \mathcal{N};$ there exist $n \geqslant 2$ different ordered pairs $(l_j, k_j)$ satisfying  \eqref{defL}, and none of them satisfies $l_j = k_j\big\}$. Then the dimension of $M$ is $2n$.
\item $\mathcal{N}_4 := \big{\{} L \in \mathcal{N};$ there exist $n \geqslant 2$ different ordered pairs $(l_j, k_j)$ satisfying  \eqref{defL},  and one of them satisfies $l_j = k_j$ $\big{\}}$.  Then  the dimension of $M$ is $2n-1$.
\end{enumerate}
The five sets $\mathcal{C}$, $\{ \mathcal{N}_i \}_{i = 1}^4$ are pairwise disjoint and
\[ \mathbb{R}^{+} = \mathcal{C} \cup \mathcal{N}_1 \cup \mathcal{N}_2 \cup \mathcal{N}_3 \cup \mathcal{N}_4,\,
\mathcal{N}= \mathcal{N}_1 \cup \mathcal{N}_2 \cup \mathcal{N}_3 \cup \mathcal{N}_4. \]
Additionally, Eduardo Cerpa proved that each of these five sets has infinite number of elements:
see \cite[Lemma 2.5]{cerpa07}; see also \cite[Proposition 8.3]{coron} for the case of $\mathcal{N}_1$.

Let us point out that $L \notin \mathcal{N} $ is equivalent to $M = \{  0 \}$. Hence, Lionel Rosier solved the (local) controllability problem of the nonlinear KdV equations for $L \in \mathcal{C}$. Later on Jean-Michel Coron and Emmanuelle Cr\'{e}peau  proved in \cite{coron04} the small-time local controllability of nonlinear KdV equations for the second case $L \in \mathcal{N}_1$, by ``power series expansion'' method,   the nonlinear term $yy_x$ gives this controllability.  Later on, in 2007, Eduardo Cerpa proved the local controllability in large time for the third case $L \in \mathcal{N}_2$ \cite{cerpa07}, still by using the ``power series expansion'' method. In this case, an expansion to the order 2 is sufficient and the local controllability in small time remains open. Finally Eduardo Cerpa  and  Emmanuelle Cr\'{e}peau in \cite{cerpa09} concluded the study by proving the local controllability in large time of \eqref{kdv} for the  two remaining critical cases (for which dim $M \geqslant 3$).  The proof of all these results rely on the ``power series expansion'' method, a method introduced in \cite{coron04}. This method has also been used to prove controllability results for Schr\"{o}dinger equations \cite{2005-Beauchard-JMPA, 2006-Beauchard-Coron-JFA, 2014-Beauchard-Morancey-MCRF, 2014-Morancey-AIHP} and for rapid asymptotic stability of a Navier-Stokes control system in \cite{2016-Chowdhury-Ervedoza-Raymond-preprint}. In this article we use
it to get exponential stabilization of \eqref{kdv}. For studies on the controllability of other KdV control systems problems, let us refer to \cite{rosier-2015-stable, Ludovick,  glass, goubet, rosier04, zhang} and the references therein.
{\color{black}}

The asymptotic stability of $0$ without control (control term equal to 0) has been studied for years, see, in particular, \cite{2013-Cerpa-Coron-IEEE, goubet, jz, mmp, pazoto, zuazua02, rz2, russz, RRz2}. Among which, for example,   the  local exponential stability for our KdV equation if $L \notin \mathcal{N}$ was proved in \cite{zuazua02}. {\color{black} Let also point out  here that in \cite{doronin2014}, the authors give the existence of (large) stationary solutions which ensures that the exponential stability result in \cite{zuazua02} is only local.}

Concerning the stabilization by means of feedback laws, the locally exponentially stabilization with arbitrary decay rate (rapid stabilization) with some linear feedback law was obtained by Eduardo Cerpa and Emmanuelle Cr\'{e}peau in \cite{2009-Cerpa-Crepau-DCDS} for the linear KdV equation \eqref{Lkdv}.
For the nonlinear case, {\color{black} the first rapid stabilization for  Korteweg-de Vries
equations  was obtained in \cite{2010-Laurent-Rosier-Zhang-CPDE} by Camille Laurent, Lionel Rosier and Bing-Yu Zhang in the case of localized distributed control on a periodic domain. In that case the linearized control system, let us write it $\dot y =Ay +Bu$, is controllable. These authors used an 
 approach due to Marshall Slemrod  \cite{1974-Slemrod-SICON} to construct linear feedback laws
  leading to the rapid stabilization of $\dot y =Ay+Bu$ and then proved that the same feedback laws give the rapid stabilization of the nonlinear Korteweg de Vries equation. In the case of distributed control the operator $B$ 
  is bounded. For boundary control the operator $B$ is unbounded. The Slemrod approach has been modified to handle this case by Vilmos Komornik in \cite{1997-Komornik} and by Jose Urquiza in \cite{2005-Urquiza}; and \cite{2009-Cerpa-Crepau-DCDS} precisely
  uses the modification presented in \cite{2005-Urquiza}. However, in contrast with the case of distributed control,  it leads to unbounded linear feedback laws and one does know for the moment if these linear feedback laws lead to asymptotic stabilization for the nonlinear Korteweg de Vries equation. One does not even know if the closed system is well posed for this nonlinear equation.} The first rapid stabilization result in the nonlinear case and with boundary controls was obtained by Eduardo Cerpa and Jean-Michel Coron in \cite{2013-Cerpa-Coron-IEEE}. Their approach relies on the backstepping method/transformation (see \cite{2008-Krstic-Smyshlyaev-book} for an excellent starting point to get inside this method due to Miroslav Krstic and his collaborators). When $L\not \in \mathcal{N}$,  by using a more general transformation and the controllability of \eqref{Lkdv} , Jean-Michel Coron and Qi L\"{u} proved in \cite{coronluqi} the rapid stabilization of our KdV control system. Their method can be applied to many other equations, like  Schr\"{o}dinger equations \cite{2016-Coron-Gagnon-Morancey-preprint} and Kuramoto-Sivashinsky equations \cite{coronluqi2}.
When $L \in \mathcal{N}$,  as mentioned above, the linearized control system \eqref{Lkdv} is not controllable, but the control system \eqref{kdv} is controllable. Let us recall that for the finite dimensional case, the controllability doesn't imply the existence of a (continuous) stationary feedback law which  {\color{black} stabilizes} (asymptotically, exponentially etc.) the control system, see \cite{74, 1990-Coron-SCL}. However the controllability in general implies the existence of (continuous) \textit{time-varying} feedback laws which asymptotically (and even in finite time) stabilize the control system; see \cite{1995-Coron}. Hence it is natural to look for  time-varying feedback laws $u(t, y(t, \cdot))$  such that $0$ is (locally) asymptotically stable for the closed-loop system
\begin{gather}\label{kdv-closed-loop}
\begin{cases}
y_t+y_{xxx}+y_x+y y_x=0 \;\;\;\;\;&\textrm{for}\;\;  (t,x) \in (s,+\infty)\times (0,L), \\
y(t,0)= y(t,L)=0   &\textrm{for}\;\; t \in (s, + \infty),\\
y_x(t,L)=u(t,y(t,\cdot))    &\textrm{for}\;\; t \in (s,+\infty).
\end{cases}
\end{gather}
{\color{black} \label{comment-LRZ-time-varying}
Let us also point out that in  \cite{2010-Laurent-Rosier-Zhang-CPDE}, as in \cite{1994-Coron-Rosier-JMSEC} by Jean-Michel Coron and Lionel Rosier which was dealing with finite dimensional control systems, time-varying feedback laws were used in order to combine two different feedback laws to get rapid \textit{global} asymptotic stability of the closed loop system.} Let us emphasize that $u = 0 $ leads to (local) asymptotic stability when $L \in \mathcal{N}_1$ \cite{coron15} and $L \in \mathcal{N}_2$ \cite{ccst}. However, in both cases, the convergence is not exponential. It is then natural to ask if we can get exponential convergence to $0$ with the help of some suitable time-varying feedback laws $u(t,y(t,\cdot))$. The aim of this paper is to
prove that it is indeed possible in the case where
\begin{equation}\label{Lgood}
\text{$L$ is in $\mathcal{N}_2$ or in $\mathcal{N}_3$.}
\end{equation}
Let us denote by $P_H: L^2(0,L)\to H$ and $P_M: L^2(0,L)\to M$ the orthogonal projection (for the $L^2$-scalar product) on $H$ and $M$ respectively. Our main result is the following one, where the precise definition of a solution of \eqref{kdv-Cauchy} is given in Section \ref{sec-preliminaries}.
\begin{theorem}\label{thm1}
Assume that \eqref{Lgood} holds. Then there exists a periodic time-varying feedback law $u$, $C>0$, $\lambda >0$ and $r > 0$ such that, for every $s \in \mathbb{R}$ and for every $\lVert y_0 \lVert_{L^2_L} < r$, the Cauchy problem
\begin{gather}\label{kdv-Cauchy}
\begin{cases}
y_t+y_{xxx}+y_x+y y_x=0 \;\;\;\;\;&\textrm{for}\;\;  (t,x) \in (s,+\infty)\times (0,L), \\
y(t,0)= y(t,L)=0   &\textrm{for}\;\; t \in (s, + \infty),\\
y_x(t,L)=u(t,y(t,\cdot))    &\textrm{for}\;\; t \in (s,+\infty),\\
y(s,\cdot) = y_0  &\textrm{for}\;\; x \in (0, L){\color{black} ,}
\end{cases}
\end{gather}
has at least one  solution in $C^0([s, +\infty); L^2(0,L)) \cap L^2_{loc}([s, +\infty); H^1(0,L))$ and every solution $y$ of \eqref{kdv-Cauchy} is defined on $[s,+\infty)$ and satisfies, for every $t\in [s, +\infty)$,
\begin{equation}
\label{cvexponentielle}
\lVert P_H (y(t)) \lVert_{L^2_L} + \lVert P_M (y(t)) \lVert_{L^2_L}^{\frac{1}{2}}
\leqslant C e^{- \lambda (t-s)} \big( \lVert P_H(y_0) \lVert_{L^2_L} + \lVert P_M (y_0) \lVert_{L^2_L}^{\frac{1}{2}} \big).
\end{equation}
\end{theorem}

In order to simplify the notations, in this paper we  sometimes simply denote $y(t, \cdot)$ by $y(t)$, if there is no misunderstanding, {\color{black} sometimes we also simply denote $L^2(0, L)$ (resp. $L^2(0, T)$) by $L^2_L$ (resp. $L^2_T$).} Let us explain briefly an important ingredient of our proof of Theorem~\ref{thm1}. Taking into account the uncontrollability
of the linearized system, it is natural to split the KdV system into a coupled system for $(P_H(y), P_M(y))$. Then the finite dimensional analogue of our KdV control system is
 \begin{equation}\label{form-quadratic}
\dot{x} = Ax + R_1 (x,y) + Bu, \;\;\;  \dot{y} = Ly + Q(x,x)+ R_2(x,y),
 \end{equation}
where $A$, $B$, and $L$ are  matrices, $Q$ is a quadratic map, $R_1, R_2$ are polynomials and $u$ is the control. The state variable $x$ plays the role of $P_H(y)$, while $y$ plays the
role of $P_M(y)$. The two polynomials $R_1$ and $R_2$ are quadratic and $R_2(x,y)$ vanishes for $y=0$. For this  ODE system, in many cases the Brockett condition \cite{74} and the Coron condition \cite{coron} for the existence of continuous stationary stabilizing feedback laws do not hold. However, as shown in \cite{2016-Coron-Rivas-preprint}, many physical systems of form \eqref{form-quadratic} can be exponentially stabilized by means of time-varying feedback laws. We follow the construction of these time-varying feedback laws given in this article. However, due to the fact that $H$ is of infinite dimension, many parts of the proof have to be modified compared to {\color{black} those} given in \cite{2016-Coron-Rivas-preprint}. In particular we do not know how to use a Lyapunov approach, in contrast to what is done in \cite{2016-Coron-Rivas-preprint}.

 This article is organized as follows. In Section \ref{sec-preliminaries},  we recall some classical results and definitions {\color{black} about}
 \eqref{kdv} and \eqref{Lkdv}. In Section~\ref{sec-closed-loop-system}, we study the existence and uniqueness of solutions to the closed-loop system \eqref{kdv-Cauchy} with  time-varying feedback laws $u$
which are not smooth. In Section~\ref{sec-contruction-feedbacks}, we  construct our time-varying feedback laws. In Section \ref{sec-proof-th1}, we prove two estimates for solutions to the closed-loop
system \eqref{kdv-Cauchy} (Propositions  \ref{proposition3} and \ref{proposition2}) which imply Theorem \ref{thm1}.  The article ends with three appendices where proofs of propositions used in the main parts of the article are given.

\section{Preliminaries}
\label{sec-preliminaries}
We first recall some results on KdV equations and give the definition of a solution to the Cauchy problem \eqref{kdv-Cauchy}.
Let us start with the nonhomogeneous linear Cauchy problem
\begin{gather}\label{likdv}
\begin{cases}
y_t+y_{xxx}+y_x = \tilde{h}   \;\;\; & \textrm{in} \;\;(T_1, T_2) \times (0, L),\\
y(t,0)= y(t,L)=0  & \textrm{on} \;\;(T_1, T_2), \\
y_x(t,L)= h(t) & \textrm{on} \;\;(T_1, T_2), \\
y(T_1,x)=y_0(x) & \textrm{on} \;\;(0, L),
\end{cases}
\end{gather}
for
\begin{gather}
\label{propertyT1T2}
-\infty< T_1<T_2<+\infty,
\\
\label{popertyy0}
y_0 \in L^2 (0,L),
\\
\label{popertytildeh}
\tilde{h} \in L^1 (T_1,T_2; L^2(0,L)),
\\
\label{popertyh}
h \in L^2(T_1,T_2).
\end{gather}
Let us now give  the definition of a solution to \eqref{likdv}.
\begin{definition}\label{definition2}
 A solution to the Cauchy problem \eqref{likdv} is a function $y \in L^1 (T_1, T_2; L^2(0,L))$  such that, for almost every $\tau \in [T_1,T_2]$ the following holds: for every
 $\phi \in C^3([T_1, \tau] \times [0,L])$ such that
\begin{equation}
\phi (t,0) = \phi (t,L) = \phi_x (t,0) = 0, \;\;\forall t \in [T_1,\tau],
\end{equation}
one has
\begin{align}
-\int_{T_1}^{\tau} \int_0^L &(\phi_t + \phi_x + \phi_{xxx})y dxdt - \int_{T_1}^{\tau} h(t) \phi_x (t,L) dt -\int_{T_1}^{\tau} \int_0^L \phi \tilde h  dxdt  \notag\\
 & + \int_0^L y(\tau,x)\phi(\tau,x) dx - \int_0^L y_0 \phi(T_1,x) dx = 0.
\end{align}
\end{definition}

 For $T_1$ and $T_2$ satisfying \eqref{propertyT1T2}, let us define the linear space $\mathcal{B}_{T_1,T_2}$ by
\begin{equation}\label{defcalB}
\mathcal{B}_{T_1,T_2} := C^0([T_1, T_2]; L^2(0,L)) \cap L^2(T_1, T_2; H^1(0,L)).
\end{equation}
This linear space $\mathcal{B}_{T_1,T_2}$ is equipped with the following norm
\begin{equation}\label{defnormB}
  \lVert y \lVert_{\mathcal{B}_{T_1,T_2}}:=\max \{\lVert y(t) \lVert_{L^2_L};\, t\in [T_1,T_2]\}
  + \left(\int_{T_1}^{T_2}\lVert y_x(t) \lVert_{L^2_L}^2 dt\right)^{1/2}.
\end{equation}
With this norm, $\mathcal{B}_{T_1,T_2}$ is a Banach space.

Let $\mathcal{A}: \mathcal{D}(\mathcal{A})\subset L^2(0,L)\rightarrow L^2(0,L)$ be the linear operator defined by
\begin{gather}\label{def-domain-A}
  \mathcal{D}(\mathcal{A}){\color{black} :=}  \left\{\phi\in H^3(0,L);\,\phi(0)=\phi(L)=\phi_x(L)=0\right\},
  \\
  \label{def-value-A}
  \mathcal{A}\phi:=-\phi_x -\phi_{xxx},\;\; \forall \phi\in \mathcal{D}(\mathcal{A}).
\end{gather}
It is known that both $\mathcal{A}$ and $\mathcal{A}^*$ are closed and dissipative  (see e.g. \cite[page 39]{coron}),  and therefore $\mathcal{A}$ generates a strongly continuous semigroup of contractions $S(t), \, t\in [0,+\infty)$ on $L^2(0,L)$.

In \cite{rosier97}, Lionel Rosier using the above properties of $\mathcal{A}$ together with multiplier techniques proved the following existence and uniqueness result for the Cauchy problem \eqref{likdv}.
\begin{lemma}\label{lem1}
The Cauchy problem \eqref{likdv} has one and only one solution. This solution is in $\mathcal{B}_{T_1,T_2}$ and there exists a constant $C_2>0$ depending only on $T_2-T_1$  such that
\begin{equation}\label{l1}
\lVert y \lVert_{\mathcal{B}_{T_1,T_2}} \leqslant C_2 \left( \lVert y_0 \lVert_{L^2_L} + \lVert h \lVert_{L^2(T_1,T_2)} + \lVert \tilde{h} \lVert_{L^1 (T_1,T_2; L^2 (0,L))} \right).
\end{equation}
\end{lemma}
In fact the notion of solution to the Cauchy problem \eqref{likdv} considered in \cite{rosier97} is a priori stronger than the one we consider here (it is required to be in $C^0([T_1,T_2];L^2(0,L))$. However the uniqueness of the solution in the sense of Definition~\ref{definition2} still follows from classical arguments; see, for example, \cite[Proof of Theorem 2.37, page 53]{coron}.

Let us now turn to the nonlinear KdV equation
\begin{gather}\label{nonlikdv}
\begin{cases}
y_t+ y_{xxx}+ y_x+ y y_x= \tilde{H}   \;\;\; & \textrm{in} \;\;(T_1, T_2) \times (0, L),\\
y(t,0)= y(t,L)=0  & \textrm{on} \;\;(T_1, T_2), \\
y_x(t,L)= H(t) & \textrm{on} \;\;(T_1, T_2), \\
y(T_1,x)=y_0(x) & \textrm{on} \;\;(0, L).
\end{cases}
\end{gather}
Inspired by Lemma~\ref{lem1}, we adopt the following definition.
\begin{definition}\label{definition3}
A solution to \eqref{nonlikdv} is a function $y\in \mathcal{B}_{T_1,T_2}$ which is a solution
of \eqref{likdv} for $\tilde h:=\tilde H-yy_x\in L^1 (T_1,T_2; L^2 (0,L))$ and $h:= H$.
\end{definition}
Throughout this article we will use similar definitions without giving them precisely.
As an example, it will be the case for system \eqref{skdv}.

In \cite{coron04}, Jean-Michel Coron and Emmanuelle Cr\'{e}peau proved the following lemma on the well-posedness
of the Cauchy problem \eqref{nonlikdv} for small initial data.
\begin{lemma}\label{lem2}
There exist $\eta >0$ and $C_3 >0 $ depending on $L$ and $T_2-T_1$ such that, for every $y_0 \in L^2 (0,L)$, every $H \in L^2(T_1,T_2)$ and every $\tilde{H} \in L^1 (T_1,T_2; L^2(0,L)) $ satisfying
\begin{equation}\label{condition}
 \lVert y_0 \lVert_{L^2_L} + \lVert H \lVert_{L^2(T_1,T_2)}
 + \lVert \tilde{H} \lVert_{L^1 (T_1,T_2; L^2 (0,L))}   \leqslant \eta,
\end{equation}
 the Cauchy problem  \eqref{nonlikdv} has a unique solution and this solution satisfies
\begin{equation}\label{nl1}
\lVert y \lVert_{\mathcal{B}_{T_1,T_2}} \leqslant C_3 \big( \lVert y_0 \lVert_{L^2_L} + \lVert H \lVert_{L^2(T_1,T_2)} + \lVert \tilde{H} \lVert_{L^1 (T_1,T_2; L^2 (0,L))} \big).
\end{equation}
\end{lemma}

\section{Time-varying feedback laws and well-posedness of the associated closed-loop system}
\label{sec-closed-loop-system}
Throughout this section $u$ denotes a time-varying feedback law: it is a map from $\R\times L^2(0,L)$ with values into $\R$. We assume that this map is a Carath\'{e}odory map, i.e.  it satisfies the three
following properties
\begin{gather}\label{locallybounded}
\forall R>0, \exists \; C_{B}(R)>0 \text{ such that } \left(\Vert y\Vert_{L^2_L} \leqslant R\Rightarrow
|u(t,y)|\leqslant C_B(R),\;\; \forall t \in \R\right),
\\
\label{measurable}
\text{ $\forall y\in L^2(0,L)$, the function $t\in \R\mapsto u(t,y)\in \R$ is measurable,}
\\
\label{almostcontinuity}
\text{for almost every  $t\in \R$, the function $y \in L^2(0,L) \mapsto u(t,y)\in \R$ is continuous.}
\end{gather}
In this article we always assume that
\begin{gather}\label{CR>1}
C_B(R) \geq 1, \;\; \forall R\in [0,+\infty), \\
\label{Cincreasing}
\text{$R\in [0, +\infty)\mapsto C_B(R)\in \R$ is a non-decreasing function.}
\end{gather}
Let $s\in \R$ and let $y_0\in L^2(0,L)$. We start by giving the definition of a solution to
\begin{gather}\label{kdv-closed-loop-I-sans-y0}
\begin{cases}
y_t+y_{xxx}+y_x+y y_x=0 \;\;\;\;\;&\textrm{for}\;\;  t\in \R, \, x\in  (0,L), \\
y(t,0)= y(t,L)=0   &\textrm{for}\;\; t\in \R ,\\
y_x(t,L)=u(t,y(t,\cdot))    &\textrm{for}\;\;  t\in \R,
\end{cases}
\end{gather}
and to the Cauchy problem
\begin{gather}\label{kdv-closed-loop-I}
\begin{cases}
y_t+y_{xxx}+y_x+y y_x=0 \;\;\;\;\;&\textrm{for}\;\;  t>s, \, x\in  (0,L), \\
y(t,0)= y(t,L)=0   &\textrm{for}\;\; t >s ,\\
y_x(t,L)=u(t,y(t,\cdot))    &\textrm{for}\;\; t>s,
\\
y(s,x)=y_0(x) &\textrm{for}\;\; x \in (0,L),
\end{cases}
\end{gather}
where $y_0$ is a given function in $L^2(0,L)$ and $s$ is a given real number.
\begin{definition}
\label{def-solution-closed-loop}
Let $I$ be an interval of $\R$ with a nonempty interior. A function $y$ is a solution of \eqref{kdv-closed-loop-I-sans-y0} on $I$ if
 $y\in C^0(I;L^2(0,L))$ is such that, for every $[T_1,T_2]\subset I$
with $-\infty< T_1<T_2<+\infty$, the restriction of $y$ to $[T_1,T_2]\times (0,L)$ is a solution of \eqref{nonlikdv} with $\tilde{H}:=0$, $H(t):=u(t,y(t))$ and $y_0:=y(T_1)$. A function $y$ is a solution to the Cauchy problem \eqref{kdv-closed-loop-I} if there exists an interval $I$ with a nonempty interior satisfying
$I\cap (-\infty,s]=\{s\}$ such that $y\in C^0(I; L^2(0,L))$  is a solution of \eqref{kdv-closed-loop-I-sans-y0} on $I$ and satisfies  the initial condition
$y(s)=y_0$ in $L^2(0,L)$. The interval $I$ is denoted  by $D(y)$.
We say that a solution $y$ to the Cauchy problem \eqref{kdv-closed-loop-I} is maximal if, for every solution $z$ to the  Cauchy problem \eqref{kdv-closed-loop-I}
such that
\begin{gather}
\label{Dzplusgrand}
  D(y)\subset D(z), \\
\label{yrestrictionofz}
y(t)=z(t) \text{ for every $t$ in } D(y),
\end{gather}
one has
\begin{equation}\label{Dy=Dz}
D(y)=D(z).
\end{equation}
\end{definition}
Let us now state our theorems concerning the Cauchy problem \eqref{kdv-closed-loop-I}.
\begin{theorem}
\label{th-feedback-Lipschitz}
Assume that $u$ is a Carath\'{e}odory function and  that, for every $R>0$, there exists $K(R)>0$ such that
\begin{equation}\label{lip-condition-R}
\left(\lVert y\rVert_{L^2_L}\leqslant R \text{ and } \lVert z\rVert_{L^2_L}\leqslant R\right)
\Rightarrow
\left(| u(t,y)-u(t,z)|\leqslant K(R) \lVert y- z\rVert_{L^2_L},\;\;\; \forall t\in \R\right).
\end{equation}
Then, for every $s\in \R$ and for every $y_0\in L^2(0,L)$, the Cauchy problem \eqref{kdv-closed-loop-I} has
one and only one maximal solution $y$. If $D(y)$ is not equal to $[s,+\infty)$, there exists $\tau\in \R$ such that
$D(y)=[s,\tau)$ and one has
\begin{equation}\label{divergenceattau}
  \lim_{t\rightarrow \tau^-}\lVert y(t)\rVert_{L^2_L }=+\infty.
\end{equation}
Moreover, if $C_B(R)$ satisfies
\begin{equation}
\label{Clinearlybounder}
  \int_0^{+\infty} \frac{R}{(C_B(R))^2}dR = + \infty,
\end{equation}
then
\begin{equation}
\label{definedpartout}
D(y)=[s,+\infty).
\end{equation}
\end{theorem}

\begin{theorem}
\label{th-feedback-not-Lipschitz}
Assume that $u$ is a Carath\'{e}odory function which satisfies condition \eqref{Clinearlybounder}.
Then, for every $s\in \R$ and for every $y_0\in L^2(0,L)$, the Cauchy problem \eqref{kdv-closed-loop-I} has
at least one maximal solution $y$ such that $D(y) = [s, + \infty)$.
\end{theorem}
The proofs of Theorem~\ref{th-feedback-Lipschitz} and  Theorem~\ref{th-feedback-not-Lipschitz} will be given in Appendix~\ref{appendix-proof-th-feedback-Lipschitz}.

We end up this section with the following proposition which gives the expected connection between  the evolution of $P_M(y)$ and $P_H(y)$
 and the fact that $y$ is a solution to \eqref{kdv-closed-loop-I-sans-y0}.
 \begin{proposition}\label{proposition1}
 Let $u: \R\times L^2(0,L)\rightarrow \R$ be a Carath\'{e}odory feedback law. Let $-\infty<s<T<+\infty$, let $y\in\mathcal{B}_{s,T}$ and let
 $y_0\in L^2(0,L)$. Let us denote $P_H(y)$ (resp. $P_M(y)$) by $y_1$ (resp. $y_2$). Then $y$ is a solution
 to the Cauchy problem \eqref{kdv-closed-loop-I} if and only if
 \begin{gather}\label{skdv}
 \left\{
 \begin{array}{l}
 \left\{
 \begin{array}{l}
 y_{1t} + y_{1x}+ y_{1xxx} + P_{H}\big((y_1 +y_2)(y_1 +y_2)_x\big) =0, \\
 y_{1}(t,0) = y_1 (t,L) = 0,\\
 y_{1x}(t,L) = u(t,y_1 +y_2),\\
 y_1(0,\cdot) = P_H(y_0),
 \end{array}
 \right.
 \\
  \left\{
  \begin{array}{l}
 y_{2t} + y_{2x}+ y_{2xxx} + P_{M}\big((y_1 +y_2)(y_1 +y_2)_x\big) =0, \\
 y_{2}(t,0) = y_2 (t,L) = 0,\\
 y_{2x}(t,L) = 0,\\
 y_2(0,\cdot) = P_M(y_0).
 \end{array}
 \right.
 \end{array}
 \right.
 \end{gather}
 \end{proposition}
 The proof of  this proposition is given in Appendix \ref{sec-appendix-A}.


\section{Construction of time-varying feedback laws}
\label{sec-contruction-feedbacks}
In this section, we construct feedback laws which will lead to the local exponential stability stated in Theorem \ref{thm1}. Let us denote by $M_{\mathbf{1}}$ the set of elements in $M$ having a $L^2$-norm equal to $1$:
\begin{equation}\label{defM1}
M_{\mathbf{1}} := \big{\lbrace} y \in M ; \, \lVert y \lVert_{L^2_L} = 1 \big{\rbrace}.
\end{equation}
Let $M^{j}$ be the linear space generated by $\varphi_1^{j}$ and $\varphi_2^{j}$ for every $j \in \{1, 2, ..., n \}$:
\begin{equation}
M^{j} := \textrm{Span} \{ \varphi_1^{j}, \varphi_2^{j}  \}.
\end{equation}
  The construction of our feedback laws relies on the following proposition.
\begin{proposition}
\label{existence-of-v}\label{prop--2}
There {\color{black} exist}  $T> 0$ and  $v\in L^{\infty} \big( [0,T] \times M_{\mathbf{1}}; \mathbb{R} \big)$ such
that the following  three properties hold.
\begin{itemize}
\item[\label{P1}($\mathcal{P}_1$)]  There exists $\rho_1 \in (0,1)$ such that
\[  \lVert S(T) y_0 \lVert_{L^2(0, L)}^{2}  \leqslant \rho_1 \lVert  y_0 \lVert_{L^2(0, L)}^{2}, \;\; \; \textrm{for every}  \;\;y_0 \in H.\]
\item[\label{P2}($\mathcal{P}_2$)]   For every $y_0 \in M $,
 \[ \lVert S(T) y_0 \lVert_{L^2(0, L)}^{2}  =  \lVert  y_0 \lVert_{L^2(0, L)}^{2}.\]
\item[\label{P3}($\mathcal{P}_3$)]  There exists $C_0 > 0$ such that
\begin{equation}\label{p3lip}
\mid v(t,y) - v(t,z) \mid \leqslant C_0 \lVert  y - z \lVert_{L^2(0, L)}, \;\;\;\; \forall t \in [0,T], \; \forall y,z \in M_{\mathbf{1}}.
\end{equation}
Moreover, there exists $\delta > 0$ such that, for every $z \in M_{\mathbf{1}}$, the solution $(y_1, y_2)$ to the following equation
\begin{gather}\label{homostab}
 \begin{cases}
 y_{1t} + y_{1x}+ y_{1xxx} =0, \\
 y_{1}(t,0) = y_1 (t,L) = 0,\\
 y_{1x}(t,L) = v(t,z),\\
 y_1(0,x) = 0, \\
 y_{2t} + y_{2x}+ y_{2xxx} + P_{M}\big(y_1 y_{1x}\big) =0, \\
 y_{2}(t,0) = y_2 (t,L) = 0,\\
 y_{2x}(t,L) = 0,\\
 y_2(0,x) = 0, \
 \end{cases}
 \end{gather}
satisfy
\begin{equation}
\label{y1nul-y2OK}
y_1 (T) = 0 \;\;\; \textrm{and}\;\;\; \big{<} y_2 (T), S(T) z \big{>}_{L^2(0, L)}  \;<\; -2\delta.
\end{equation}
\end{itemize}
\end{proposition}
\begin{proof}[Proof of Proposition~\ref{existence-of-v}]
Property \hyperref[P2]{($\mathcal{P}_2$)} is given in \cite{rosier97}, one can also see  \eqref{SrotationonM} and \eqref{SrotationonM^j}. Property \hyperref[P1]{($\mathcal{P}_1$)} follows from the dissipativity of $\mathcal{A}$ and the controllability of
\eqref{Lkdv} in $H$ (see also \cite{zuazua02}). Indeed, integrations by parts (and simple density arguments) show that,
in the distribution sense in $(0,+\infty)$,
\begin{equation}\label{coro1}
\frac{d}{d t} \lVert S(t)y_0 \lVert_{L^2_{L}}^2 = - y_x^2 (t,0).
\end{equation}
Moreover, as Lionel Rosier proved in \cite{rosier97}, for every $T>0$, there exists $c> 1$ such that, for every $y_0 \in H$,
\begin{equation}\label{coro2}
\lVert y_0 \lVert_{L^2_{L}}^2 \leqslant c \lVert y_x(t, 0) \lVert_{L^2(0, T)}^2.
\end{equation}
Integration of identity \eqref{coro1} on $(0,T)$ and the use of \eqref{coro2} give
\begin{equation}
\lVert S(T)y_0 \lVert_{L^2_L}^2 \leqslant  \frac{c-1}{c}\lVert y_0 \lVert_{L^2_L}^2.
\end{equation}
Hence $\rho _1:=(c-1)/c\in (0,1) $ satisfies the required properties.

Our concern now is to deal with \hyperref[P3]{($\mathcal{P}_3$)}.  Let us first recall a result on
the controllability of the linear control system
\begin{gather}\label{Lkdv-new}
\begin{cases}
y_t+y_{xxx}+y_x=0 \;\;\;\;\;&\textrm{in}\;\;  (0,T)\times (0,L), \\
y(t,0)= y(t,L)=0   &\textrm{on}\;\; (0,L),\\
y_x(t,L)=u(t)    &\textrm{on}\;\; (0,T),\\
\end{cases}
\end{gather}
where, at time $t\in[0,T]$  the state is $y(t, \cdot) \in L^2(0,L)$. Our goal is to investigate the cases where $L\in \mathcal{N}_2\cup \mathcal{N}_3$, but in order to explain more clearly our construction of $v$, we first deal with the case where
\begin{equation}\label{choice-L}
L = 2 \pi \sqrt{\frac{1^2 + 1\times 2 + 2^2}{3}}=2 \pi \sqrt{\frac{7}{3}},
\end{equation}
which corresponds to $l=1$ and $k=2$ in \eqref{N}.
In that case the uncontrollable subspace  $M$ is a two dimensional vector subspace of $L^2(0,L)$  generated by
\begin{gather}
\varphi_1 (x)= C\left(\cos\left(\frac{5}{\sqrt{21}}x\right)-3\cos\left(\frac{1}{\sqrt{21}}x\right)+2\cos\left(\frac{4}{\sqrt{21}}x\right)\right),\notag\\
\varphi_2 (x)= C\left(-\sin\left(\frac{5}{\sqrt{21}}x\right)-3\sin\left(\frac{1}{\sqrt{21}}x\right)+2\sin\left(\frac{4}{\sqrt{21}}x\right)\right), \notag
\end{gather}
 where $C$ is a positive constant such that $\lVert \varphi_1 \lVert_{L^2_L} = \lVert \varphi_2 \lVert_{L^2_L} = 1$. They satisfy
\begin{gather}\label{varphi}
\begin{cases}
\displaystyle  \varphi_1' + \varphi_1''' = -\frac{2\pi}{p} \varphi_2, \\
\varphi_1(0) = \varphi_1 (L) = 0, \\
\varphi_1' (0) = \varphi_1' (L) = 0,
\end{cases}
\end{gather}
and
\begin{gather}\label{varphi2}
\begin{cases}
\displaystyle \varphi_2' + \varphi_2''' =\frac{2\pi}{p} \varphi_1, \\
\varphi_2(0) = \varphi_2 (L) = 0, \\
\varphi_2' (0) = \varphi_2' (L) = 0,
\end{cases}
\end{gather}
with  (see  \cite{cerpa07})
\begin{equation}\label{valuep}
p:= \frac{441\pi }{10 \sqrt{21}}.
\end{equation}
 For every $t>0$, one has
\begin{equation}\label{SrotationonM}
S(t)M\subset M \text{ and  $S(t)$ restricted to $M$ is the rotation  of angle } \frac{ 2\pi t}{p},
\end{equation}
if the orientation on $M$ is chosen so that $(\varphi_1,\varphi_2)$ is a direct basis,
a choice which is done from now on.  Moreover the control $u$ has no action on $M$ for the linear control system \eqref{Lkdv}:
{\color{black} for every initial data $y_0\in M$, whatever is $u\in L^2(0,T)$, the solution $y$ of \eqref{Lkdv} with $y(0)=y_0$ satisfies $P_M( y(t))= S(t)y_0$, for every $t\in[0,+\infty)$.}
Let us denote by $H$ the orthogonal in $L^2(0,L)$ of $M$ for the $L^2$-scalar product $H:= M^{\perp}$. This linear space
is left invariant by the linear control system \eqref{Lkdv}: for every initial data $y_0\in H$, whatever is $u\in L^2(0,T)$, the solution $y$ of \eqref{Lkdv} satisfying $y(0)=y_0$ is such that $y(t)\in H$, for every $t\in[0,+\infty)$. Moreover,
as it is proved by Lionel Rosier  in \cite{rosier97},  the linear control system \eqref{Lkdv} is controllable in $H$ in small-time. More precisely, he proved the following lemma.
\begin{lemma}\label{lem4}
Let $T>0$. There exists $C>0$ depending only on $T$  such that, for every $y_0$, $y_1 \in H$, there exists {\color{black} a} control $u \in L^2(0,T)$ satisfying
\begin{equation}
\lVert u\lVert_{L^2_T} \leqslant C ( \lVert y_0 \lVert_{L^2_L} + \lVert y_1 \lVert_{L^2_L} ),
\end{equation}
such that the solution $y$ of the Cauchy problem
\begin{gather*}
\begin{cases}
y_t+y_{xxx}+y_x = 0 \;\;\; & \textrm{in} \;\;(0, T) \times (0, L),\\
y(t,0)= y(t,L)=0 & \textrm{on} \;\;(0, T), \\
y_x(t,L)= u(t) & \textrm{on} \;\;(0, T),
\\
y(0,x) = y_0(x) & \textrm{on} \;\;(0, L),
\end{cases}
\end{gather*}
 satisfies $y(T,\cdot) = y_1$.
\end{lemma}

 A key ingredient of our construction of $v$  is the following proposition.
\begin{proposition}\label{prop-order-2-H1-au-lieu-de-L2}
Let $T > 0$. For every $L \in \mathcal{N}_2 \cup \mathcal{N}_3$, for every $j \in \{1, 2, ..., n \}$, there exists $u^j \in H^1 (0,T)$ such that
\begin{equation*}
 \alpha(T, \cdot) = 0 \;\; \;\textrm{and} \;\;\;  P_{M^{j}} (\beta (T, \cdot)) \neq  0,
\end{equation*}
where $\alpha$ and $\beta$ are the solution of
\begin{gather}\label{cerab}
 \begin{cases}
 \alpha_{t} + \alpha_{x}+ \alpha_{xxx} =0, \\
 \alpha(t,0) = \alpha (t,L) = 0,\\
 \alpha_{x}(t,L) = {\color{black} u^j(t)},\\
 \alpha(0,x) = 0, \\
 \beta_{t} + \beta_{x}+ \beta_{xxx} + \alpha \alpha_x  =0, \\
 \beta(t,0) = \beta (t,L) = 0,\\
 \beta_x(t,L) = 0,\\
 \beta(0,x) = 0. \
 \end{cases}
 \end{gather}
\end{proposition}
Proposition \ref{prop-order-2-H1-au-lieu-de-L2} is due to Eduardo Cerpa and Emmanuelle Cr\'{e}peau if one requires only $u$ to be in $L^2(0,T)$ instead of being in $H^1(0,T)$: see \cite[Proposition 3.1]{cerpa07} and \cite[Proposition 3.1]{cerpa09}. We explain in Appendix~\ref{sec-appendix-C} how to modify the proof of \cite[Proposition 3.1]{cerpa07}  (as well as \cite[Proposition 3.1]{cerpa09}) in order to get Proposition \ref{prop-order-2-H1-au-lieu-de-L2}.

We decompose $\beta$ by $\beta = \beta_1 + \beta_2$, where $\beta_1 := P_H (\beta)$ and $\beta_2:= P_M(\beta)$. Hence, similarly to   Proposition~\ref{proposition1}, we get
\begin{gather}\label{cerb2}
 \begin{cases}
 \beta_{2t} + \beta_{2x}+ \beta_{2xxx} +P_M \big( \alpha \alpha_x \big) =0, \\
 \beta_2(t,0) = \beta_2 (t,L) = 0,\\
 \beta_{2x}(t,L) = 0,\\
 \beta_2(0,x) = 0,
 \end{cases}
\end{gather}
where $\beta_2 (T,\cdot) = P_M (\beta (T, \cdot)) \neq  0$. In particular,  $P_{M^j} (\beta_2 (T,\cdot)) = P_{M^j} (\beta (T, \cdot)) \neq  0$.\\

 Combining \eqref{cerab} and \eqref{cerb2}, we get:
\begin{corollary}\label{corollary3}
For every $L \in \mathcal{N}_2 \cup \mathcal{N}_3$, for every $ T_0 > 0$, for every $j \in \{1, 2, ..., n\}$,  there exists $u_0^j \in L^{\infty} (0,T_0)$ such that the solution $(y_1,y_2)$ to equation \eqref{homostab} with $v(t,z) := u_0^j(t)$ satisfies
\begin{equation}\label{first}
y_1 (T_0) = 0 \;\;\;\; \textrm{and} \;\;\;\; P_{M^j} (y_2(T_0)) \neq 0.
\end{equation}
\end{corollary}

Now we come back to the case when \eqref{choice-L} holds.  Let us fix $T_0>0$ such that
\begin{equation}\label{T0<p/4}
T_0<\frac{p}{4}.
\end{equation}
Let
\begin{equation}\label{defq}
q:=\frac{p}{4}.
\end{equation}
Let $u_0$ be as in Corollary \ref{corollary3}.  We denote by
\begin{equation}
Y_1(t):= y_1(t), \;\;Y_2(t):= y_2(t), \textrm{ for } t\in  [0, T_0],
\end{equation}
and
\begin{equation}\label{defpsi1}
\psi_1:=Y_2(T_0) \in M\setminus\{0\}.
\end{equation}
Let
\begin{gather}
\label{defpsii}
\psi_2=S(q)\psi_1\in M,\;\;\; \psi_3=S(2q)\psi_1\in M,\;\;\;\psi_4=S(3q)\psi_1 \in M,
\\
\label{defT}
  T:=3q+T_0,
\\
\label{defK1}
K_1 := [3q,3q+ T_0],
\\
\label{defK2}
K_2 := \left[2q,2q+ T_0\right],
\\
\label{defK3}
K_3 := \left[q,q+ T_0\right],
\\
\label{defK4}
K_4 := \left[0,T_0\right].
\end{gather}
Note that \eqref{T0<p/4} implies that
\begin{equation}\label{disjoint}
\text{$K_1$, $K_2$, $K_3$ and $K_4$ are pairwise disjoint.}
\end{equation}
Let us define four functions $[0,T]\rightarrow \R$:   $u_{1}$, $u_{2}$, $u_{3}$ and $u_4$ by requiring that, for every $i\in\{1,2,3,4\}$,
\begin{gather}
u_{i} :=
\left \{
\begin{array}{ll}
 0 \;\; & \text{ on } [0,T]\setminus K_i, \\
u_0(\cdot - \tau_i) &   \text{ on }K_i,
\end{array}
\right.
\end{gather}
with
\begin{equation}\label{deftaui}
\tau_1=3q, \,\tau_2=2q, \, \tau_3=q,\tau_4=0.
\end{equation}
One can easily verify that, for every $i \in \{1, 2, 3, 4\}$, the solution of \eqref{homostab} {\color{black} for $v = u_i$} is given explicitly by
\begin{gather}\label{yi1tki}
y_{i, 1} (t) =
\left \{
\begin{array}{ll}
0 \;\; & \text{ on } [0,T]\setminus K_i, \\
Y_1(\cdot - \tau_i) &   \text{ on }K_i,
\end{array}
\right.
\end{gather}
and
\begin{gather}\label{yi2tki}
y_{i, 2} (t) =
\left \{
\begin{array}{ll}
0 \;\; & \text{ on } [0,\tau_i], \\
Y_2(\cdot - \tau_i) &   \text{ on }K_i,\\
S(\cdot- \tau_i- T_0) \psi_1 &  \text{ on }[\tau_i + T_0, T].
\end{array}
\right.
\end{gather}

For $z\in M_1$, let $\alpha_1$, $\alpha_2$, $\alpha_3$ and $\alpha_4$ in $[0,+\infty)$ be such that
\begin{gather}
\label{S(-T)}
-S(T)z = \alpha_1\psi_1+ \alpha_2\psi_2+\alpha_3\psi_3+ \alpha_4\psi_4,
\\
\alpha_1 \alpha_3=0,\, \alpha_2\alpha_4=0.
\end{gather}
Let us  define
\begin{equation}
\label{defv}
{\color{black} v(t, z):= \alpha_1 u_1(t)+ \alpha_2 u_{2}(t)+ \alpha_{3} u_{3}(t)+ \alpha_{4} u_{4}(t).}
\end{equation}
We notice that
\begin{equation}\label{alphasum}
(\alpha_1^2+ \alpha_2^2+ \alpha_3^2+ \alpha_4^2) \lVert \psi_1 \lVert_{L^2_{L}}^2 = 1,
\end{equation}
{\color{black} which, together with \eqref{defv}, implies that}
\begin{equation}
v \in L^{\infty} ([0, T] \times M_1; \mathbb{R}).
\end{equation}
Moreover, using the above construction (and in particular \eqref{disjoint}), one easily  checks that the solution of \eqref{homostab} satisfies
\begin{gather}
\label{y1(t)}
y_1(t) = \alpha_1 y_{1, 1}(t) + \alpha_2 y_{2, 1}(t)+ \alpha_3 y_{3, 1}(t)+ \alpha_4 y_{4, 1}(t), \text{ for } t \in [0, T],
\\
\label{y2(t)}
y_2(t) = \alpha_1^2 y_{1, 2}(t) + \alpha_2^2 y_{2, 2}(t)+ \alpha_3^2 y_{3, 2}(t)+ \alpha_4^2 y_{4, 2}(t), \text{ for } t \in [0, T].
\end{gather}
In particular
\begin{gather}
\label{y1(T)=0}
y_1(T) = 0,
\\
\label{y2(T)bon}
y_2(T) = \alpha_1^2 \psi_1+ \alpha_2^2 \psi_2 +\alpha_3^2 \psi_3+\alpha_4^2 \psi_4.
\end{gather}
From \eqref{S(-T)}, \eqref{alphasum} and \eqref{y2(T)bon}, we can find that  \eqref{y1nul-y2OK} holds if $\delta>0$ is small enough. It is easy to check that the Lipschitz condition \eqref{p3lip} is also satisfied.  This completes the construction of $v(t,z)$ such that  \hyperref[P3]{($\mathcal{P}_3$)} holds and also the proof of
Proposition~\ref{existence-of-v} if \eqref{choice-L} holds.

 For other values of $L\in \mathcal{N}_2$, only the values of $\varphi_1$, $\varphi_2$ and $p$ {\color{black} have} to be modified. For $L\in \mathcal{N}_3$, as mentioned in the introduction, $M$ is now of dimension $2n$ where $n$ is the number of ordered pairs. It is proved in \cite{cerpa09} that (compare with \eqref{varphi}--\eqref{SrotationonM}), by a good choice of order on $\{ \varphi^j \}$ one can assume
\begin{equation}\label{pdiff}
0 < p^1 < p^2 < ... < p^n,
\end{equation}
where $p^j := 2\pi / \omega^j$. For every $t > 0$, one has
\begin{equation}\label{SrotationonM^j}
S(t)M^j \subset M^j \text{ and  $S(t)$ restricted to $M^j$ is the rotation  of angle } \frac{ 2\pi t}{p^j}.
\end{equation}

From \eqref{pdiff}, \eqref{SrotationonM^j} and Corollary \ref{corollary3}, one can get the following corollary (see also \cite[Proposition 3.3]{cerpa09}):
\begin{corollary}\label{coro-only M^j not 0}
For every $L \in \mathcal{N}_3$, there exists $T_L > 0$ such that, for every $j \in \{1, 2, ..., n\}$,  there exists $u_0^j \in L^{\infty} (0,T_L)$ such that the solution $(y_1,y_2)$ to equation \eqref{homostab} with $v(t,z) := u_0^j(t)$ satisfies
\begin{equation}
y_1 (T_L) = 0 \;\;\;\; \textrm{and} \;\;\;\; y_2(T_L) = \varphi_1^j.
\end{equation}
\end{corollary}

Let us define
\begin{equation}
\psi_1^j := \varphi_1^j, \;\;\;\psi_2^j := S(q^j) \varphi_1^j, \;\;\;\psi_3^j := S(2q^j) \varphi_1^j,\;\;\; \psi_4^j := S(3q^j) \varphi_1^j,
\end{equation}
where $q^{j}:= p^{j}/4$.

Compare with \eqref{defpsi1}--\eqref{yi2tki}, we can find $T > T_L$ and closed interval sets $\{ K^j_i \}$, where $i \in \{1, 2, 3, 4\}$ and $j \in \{1, 2, ..., n\}$, such that
\begin{gather}
K^j_i \subset [0, T], \\
\{  K_i^j   \} \text{ are pairwise disjoint.}
\end{gather}
We can also find functions $\{ u^j_i \} \in L^{\infty}([0, T]; \mathbb{R})$, with
\begin{equation}
u^j_i(t) \text{ supports on } K^j_i,
\end{equation}
such that when we define the control as $u^j_i$, we get the solution of \eqref{homostab} satisfies
\begin{gather}
y^j_{i, 1}(t) \text{ supports on } K^j_i, \\
y^j_{i, 1}(T) = 0,\\
y^j_{i, 2}(T) = \psi_i^j.
\end{gather}

Then for $z \in M_1$, let $\alpha^j_i$ in $[0, + \infty)$ be such that (where $i \in \{1, 2, 3, 4\}$ and $j \in \{1, 2, ..., n\}$)
\begin{gather}
-S(T)z = \sum_{i, j} \alpha^j_i \psi^j_i, \\
\alpha^j_1 \alpha^j_3= 0, \;\;\;\alpha^j_2 \alpha^j_4= 0,\;\;\;    \sum_{i, j} (\alpha^j_i)^2 = 1.
\end{gather}
Let us define
\begin{equation}
v(t, z):= \sum_{i, j} \alpha_i^j u_i^j(t).
\end{equation}
Then the solution of \eqref{homostab} with control defined as $v(t, z)$ satisfies
\begin{gather}
y_1(T) = 0,\\
y_2(T) = \sum_{i, j} (\alpha_i^j)^2 \psi^j_i.
\end{gather}
One can easily verify that condition \eqref{y1nul-y2OK} holds when $\delta > 0$ is small enough, and that Lipschitz condition \eqref{p3lip} also holds. {\color{black} This completes the construction of $v(t, z)$ and the proof of Proposition~\ref{prop--2}.}
\end{proof}

We are now able to define the periodic time-varying feedback laws $u_{\varepsilon} : \mathbb{R} \times L^2(0,L) \rightarrow \mathbb{R}$,  which will lead to the exponential stabilization of \eqref{kdv}.  For $\varepsilon > 0$, we define $u_{\varepsilon}$ by
\begin{equation}\label{u}
u_{\varepsilon}\big{|}_{[0,T) \times L^2_L} (t,y) :=
\left \{  \begin{array}{ll}
0 & \textrm{if} \;\;\lVert y^M \lVert_{L^2_L} =0, \\
\varepsilon \sqrt{\lVert y^M \lVert_{L^2_L}} v\big( t, \frac{S(-t)y^M}{\lVert y^M \lVert_{L^2_L}} \big)\;\;\;\;\;  & \textrm{if}\;\; 0 < \lVert y^M \lVert_{L^2_L} \leqslant 1, \;\;\;\;\;\;\;\;\;\;\;\;\;\;\;\;\\
\varepsilon  v\big( t, \frac{S(-t)y^M}{\lVert y^M \lVert_{L^2_L}} \big) & \textrm{if} \;\;\lVert y^M \lVert_{L^2_L}  > 1,
\end{array} \right.
\end{equation}
with $y^M:= P_M(y)$,  and
\begin{equation}\label{uu}
u_{\varepsilon}(t,y) := u_{\varepsilon}\big{|}_{[0,T) \times L^2_L}  (t-[\frac{t}{T}]T,  y), \;\;\; \forall t \in \mathbb{R},\;\;\; \forall y \in L^2(0, L).
\end{equation}

\section{Proof of Theorem \ref{thm1}}
\label{sec-proof-th1}

Let us first point out that Theorem~\ref{thm1} is a consequence of  the  following two propositions.
\begin{proposition}\label{proposition3}
There exist $\varepsilon_1 >0$, $r_1 >0$ and $C_1$ such that, for every Carath\'{e}odory feedback law $u$ satisfying
\begin{equation}\label{utye}
 |u(t,z)| \leqslant \varepsilon_1 \textrm{min} \{ 1, \sqrt{\lVert P_M (z) \lVert_{L^2_L}} \},
 \, \;\;\forall t\in \R,\, \forall z \in L^2(0,L),
\end{equation}
for every $s\in \R$ and for every maximal solution $y$ of \eqref{kdv-closed-loop-I-sans-y0} defined at time $s$  and satisfying
$\lVert y(s) \lVert_{L^2_L} < r_{1} $, $y$ is well-defined on $[s,s+T]$ and one has
\begin{equation}\label{aprio}
\lVert  P_H(y)\lVert^2_{\mathcal{B}_{s,s+T}} +\lVert  P_M(y)\lVert_{\mathcal{B}_{s,s+T}} \leqslant  C_1 (\lVert P_H(y(s)) \lVert^2_{L^2_L} +\lVert P_M(y(s)) \lVert_{L^2_L}).
\end{equation}
\end{proposition}

\begin{proposition}\label{proposition2}
 {\color{black} For $ \rho_1$ as in Proposition~\ref{prop--2},  let $\rho_2 > \rho_1$.} There exists $\varepsilon_0 \in (0,1) $ such that, for every $\varepsilon \in (0, \varepsilon_0)$, there exists $r_{\varepsilon}>0$ such that,  for every  solution $y$
 to \eqref{kdv-closed-loop-I-sans-y0} on $[0,T]$,  for the feedback law $u:=u_{\varepsilon}$ defined in \eqref{u} and \eqref{uu}, and  satisfying $\lVert y(0) \lVert_{L^2_L} < r_{\varepsilon} $, one has
\begin{equation}\label{maines}
\lVert P_H(y(T)) \lVert^2_{L^2_L} +\varepsilon \lVert P_M (y(T)) \lVert_{L^2_L}  \leqslant \rho_2
\lVert P_H(y(0)) \lVert^2_{L^2_L} +\varepsilon (1-\delta \varepsilon^2) \lVert P_M( y(0)) \lVert_{L^2_L}.
\end{equation}
\end{proposition}
{\color{black} Indeed, it suffices to choose  $\rho_2 \in (\rho_1,1)$,  $\varepsilon \in (0,\varepsilon_0)$ and $u:=u_{\varepsilon}$ defined in \eqref{u} and \eqref{uu}.} Then, using the $T$-periodicity of $u$ with respect to time, Proposition~\ref{proposition3} and Proposition~\ref{proposition2}, one  checks
that inequality \eqref{cvexponentielle} holds with $\lambda:=\text{ min } \{ -(\ln (\rho_2))/(2T), -(\ln(1-\delta \varepsilon^2))/(2T) \}$  provided that $C$ is large enough and that $r$ is small enough.  We now prove  Proposition~\ref{proposition3} and Proposition~\ref{proposition2} successively.

\begin{proof}[Proof of Proposition~\ref{proposition3}]
Performing a time-translation if necessary, we may assume without loss of generality that  $s=0$. The fact that the maximal solution $y$ is at least defined on $[0,T]$  follows from Theorem \ref{th-feedback-not-Lipschitz} and \eqref{utye}. We choose $\varepsilon_1 $ and $r_1$ small enough so that
\begin{equation}\label{epsilon1r1petit}
  r_1 + \varepsilon_1 T^{\frac{1}{2}} \leqslant \eta,
\end{equation}
where $\eta>0$ is as in Lemma~\ref{lem2}. From \eqref{utye} and \eqref{epsilon1r1petit}, we have
\begin{equation}\label{hypp-lemma-2-satisfied}
  \lVert y(0) \lVert_{L^2_L} + \lVert u(t, y(t)) \lVert_{L^2_T}   \leqslant \eta,
\end{equation}
which allows to apply Lemma \ref{lem2} with $H(t) := u(t, y(t))$, $\tilde{H} := 0$. Then, using \eqref{utye} once more,  we get
\begin{align}
\lVert y \lVert_{\mathcal{B}} & \leqslant C_3 \big( \lVert y_0 \lVert_{L^2_L} + \lVert u(t, y(t)) \lVert_{L^2_T} \big)  \notag \\
 & \leqslant C_3 \big(  r_1 + \varepsilon_1 \sqrt{T\lVert P_M(y) \lVert_{C^0L^2_L}} \big) \notag\\
 & \leqslant C_3 \big(  r_1 + \varepsilon_1^2 T C_3 + \frac{1}{4 C_3} \lVert y \lVert_{\mathcal{B}} \big), \notag\
\end{align}
which implies that
\begin{equation}\label{123}
\lVert y \lVert_{\mathcal{B}}  \leqslant 2 C_3 \big(  r_1 + \varepsilon_1^2 T C_3 \big).
\end{equation}
In the above inequalities and until the end of the proof of Proposition~\ref{proposition2},  $\mathcal{B}:=\mathcal{B}_{0,T}$.

We have the following lemma, see the proof of \cite [Proposition 4.1 and (4.14)]{rosier97} or \cite[page 121]{zuazua02}.
 \begin{lemma}\label{lem7}
If $y\in L^2(0,T; H^1(0,L))$, then $yy_x \in L^1(0,T; L^2(0,L))$. Moreover, there exists $c_4>0$, which is independent of $T$, such that, for every $T>0$ and for every $y,z \in L^2(0,T; H^1(0,L))$, we have
\begin{equation}\label{c4}
\lVert y y_x - z z_x \lVert_{L^1_TL^2_L} \leqslant  c_4 T^{\frac{1}{4}} \big(   \lVert y \lVert_{L^2_TH^1_L}   + \lVert z \lVert_{L^2_TH^1_L}    \big) \lVert y-z \lVert_{L^2_TH^1_L}.
\end{equation}

 \end{lemma}

Let us define $C_4:= c_4 T^{\frac{1}{4}}$. To simplify the notations, until the end of this section,
we write $y_1$ and $y_2$ for $P_H(y)$ and $P_M(y)$ respectively. From  \eqref{utye},  \eqref{123},
Lemma \ref{lem1}, Lemma \ref{lem7} and Proposition \ref{proposition1}, we get
\begin{align}\label{1234}
\lVert y_1 \lVert_{\mathcal{B}} & \leqslant C_2 \big( \lVert y_0^H \lVert_{L^2_L} +  \lVert u(t,y_1 +y_2)\lVert_{L^2_T}  +\lVert P_H\big( (y_1 + y_2)(y_1 +y_2)_x \big) \lVert_{L^1_TL^2_L}  \big)  \notag\\
& \leqslant  C_2 \big( \lVert y_0^H \lVert_{L^2_L} +  \varepsilon_1 \lVert \sqrt{\lVert y_2 \lVert_{L^2_L}}\lVert_{L^2_T}  +\lVert  (y_1 + y_2)(y_1 +y_2)_x  \lVert_{L^1_TL^2_L}  \big) \notag \\
& \leqslant C_2 \big(\lVert y_0^H \lVert_{L^2_L} +
 \varepsilon_1 \lVert y_2 \lVert_{L^1_TL^2_L}^{\frac{1}{2}} + C_4 \lVert y_1 +y_2 \lVert_{L^2_TH^1_L}^2    \big),
\end{align}
and
\begin{align}\label{12345}
\lVert y_2 \lVert_{\mathcal{B}} & \leqslant C_2 \big( \lVert y_0^M \lVert_{L^2_L}  +\lVert P_M\big( (y_1 + y_2)(y_1 +y_2)_x \big) \lVert_{L^1_TL^2_L}  \big)  \notag\\
& \leqslant  C_2 \big( \lVert y_0^M \lVert_{L^2_L}  +\lVert  (y_1 + y_2)(y_1 +y_2)_x  \lVert_{L^1_TL^2_L}  \big) \notag  \\
& \leqslant  C_2 \big( \lVert y_0^M \lVert_{L^2_L}  +C_4 \lVert y_1 +y_2 \lVert_{L^2_TH^1_L}^2  \big)  \notag \\
& \leqslant 2 C_2 \big( \lVert y_0^M \lVert_{L^2_L}  +C_4 \lVert y_1  \lVert_{\mathcal{B}}^2 + C_4 \lVert y_2 \lVert_{\mathcal{B}}^2 \big).
\end{align}
Since $M$ is a finite dimensional subspace of $H^1(0,L)$, there exists $C_5>0$ such that
\begin{equation}\label{c5}
 \lVert f \lVert_{H^1(0,L)} \leqslant C_5 \lVert f \lVert_{L^2_L},\;\;\;\;\;\textrm{for every} \;f \in M.
 \end{equation}
Hence
\begin{align}\label{21}
\lVert y_2 \lVert_{\mathcal{B}} & = \lVert y_2 \lVert_{L^{\infty}_TL^2_L} +\lVert y_2 \lVert_{L^2_TH^1_L}  \notag\\
& \leqslant \lVert y_2 \lVert_{L^{\infty}_TL^2_L}  + C_5 \sqrt{T}\lVert y_2 \lVert_{L^{\infty}_TL^2_L}.
\end{align}
Since $y_2(t)$ is the $L^2$-orthogonal projection on $M$ of $y(t)$, we have
\begin{equation*}
\lVert y_2 \lVert_{L^{\infty}_TL^2_L} \leqslant \lVert y \lVert_{L^{\infty}_TL^2_L} \leqslant \lVert y \lVert_{\mathcal{B}},
\end{equation*}
which, together with \eqref{123} and \eqref{21}, implies that
\begin{equation}\label{222}
\lVert y_2 \lVert_{\mathcal{B}} \leqslant (1 + C_5 \sqrt{T} ) \lVert y \lVert_{\mathcal{B}} \leqslant 2 (1 + C_5 \sqrt{T} ) C_3 \big(  r_1 + \varepsilon_1^2 T C_3 \big).
\end{equation}
Decreasing if necessary  $r_1 $ and $\varepsilon_1$, we may assume that

\begin{equation}\label{condition2}
4 C_2 C_4  (1 + C_5 \sqrt{T} ) C_3 \big(  r_1 + \varepsilon_1^2 T C_3 \big)  < \frac{1}{2}.
\end{equation}
 From estimation \eqref{12345} and condition \eqref{condition2}, we get that
\begin{align}\label{321}
\lVert y_2 \lVert_{\mathcal{B}}  \leqslant  4 C_2 \big( \lVert y_0^M \lVert_{L^2_L}  +C_4 \lVert y_1  \lVert_{\mathcal{B}}^2 \big).
\end{align}
From   \eqref{123},  \eqref{1234}, \eqref{222} and \eqref{321}, we  deduce that
\begin{align}\label{4321}
\lVert y_1 \lVert_{\mathcal{B}}^2
& \leqslant 3 C_2^2 \big(\lVert y_0^H \lVert_{L^2_L}^2 + \varepsilon_1^2 \lVert y_2 \lVert_{L^1_TL^2_L} + C_4^2 \lVert y_1 +y_2 \lVert_{L^2_TH^1_L}^4    \big) \notag\\
& \leqslant 3 C_2^2 \big(\lVert y_0^H \lVert_{L^2_L}^2 + \varepsilon_1^2 T\lVert y_2  \lVert_{L^{\infty}_TL^2_L} + 2C_4^2 \lVert y \lVert_{\mathcal{B}}^2 (\lVert y_1 \lVert_{\mathcal{B}}^2 +\lVert y_2 \lVert_{\mathcal{B}}^2)   \big) \notag \\
& \leqslant 3 C_2^2 \lVert y_0^H \lVert_{L^2_L}^2  +3 C_2^2 \big( \varepsilon_1^2 T + 16 C_4^2 (1 + C_5 \sqrt{T}) C_3^3 (r_1+\varepsilon_1^2 T C_3)^3\big)\lVert y_2  \lVert_{\mathcal{B}}  \notag\\
   & \;\;\;\;\;\; \;\;\;\;\;\;\;\;\;\; +  24C_2^2 C_4^2  C_3^2 (r_1+\varepsilon_1^2 T C_3)^2 \lVert y_1 \lVert_{\mathcal{B}}^2 \notag \\
   & \leqslant 3 C_2^2 \lVert y_0^H \lVert_{L^2_L}^2  +12 C_2^3 \big( \varepsilon_1^2 T + 16 C_4^2 (1 + C_5 \sqrt{T}) C_3^3 (r_1+\varepsilon_1^2 T C_3)^3\big)\lVert y_0^M  \lVert_{L^2_L}  \notag\\
   &  +  \Big( 12 C_2^3 C_4 \big( \varepsilon_1^2 T + 16 C_4^2 (1 + C_5 \sqrt{T}) C_3^3 (r_1+\varepsilon_1^2 T C_3)^3\big)    + 24C_2^2 C_4^2  C_3^2 (r_1+\varepsilon_1^2 T C_3)^2  \Big) \lVert y_1 \lVert_{\mathcal{B}}^2  \notag. \\
\end{align}
Again, decreasing if necessary  $r_1 $ and $\varepsilon_1$,
we may assume that
\begin{equation}\label{condition3}
 12 C_2^3 C_4 \big( \varepsilon_1^2 T + 16 C_4^2 (1 + C_5 \sqrt{T}) C_3^3 (r_1+\varepsilon_1^2 T C_3)^3\big)    + 24 C_2^2 C_4^2  C_3^2 (r_1+\varepsilon_1^2 T C_3)^2  < \frac{1}{2}.
\end{equation}
 From \eqref{4321} and \eqref{condition3}, we get
\begin{align*}
\lVert y_1 \lVert_{\mathcal{B}}^2
& \leqslant 6 C_2^2 \lVert y_0^H \lVert_{L^2_L}^2 + 24 C_2^3 \big( \varepsilon_1^2 T + 16 C_4^2 (1 + C_5 \sqrt{T}) C_3^3 (r_1+\varepsilon_1^2 T C_3)^3\big)\lVert y_0^M  \lVert_{L^2_L} \notag\\
& \leqslant 6 C_2^2 \lVert y_0^H \lVert_{L^2_L}^2 + C_4^{-1}\lVert y_0^M  \lVert_{L^2_L}, \notag
\end{align*}
which, combined with  \eqref{321}, gives the existence of $C_1>0$ independent of $y$ such that
\begin{equation}\label{apfinal}
\lVert y_1 \lVert_{\mathcal{B}}^2  +\lVert y_2  \lVert_{\mathcal{B}}  \leqslant C_1 \big(   \lVert y_0^H \lVert_{L^2_L}^2  +\lVert y_0^M  \lVert_{L^2_L}   \big).
\end{equation}
This completes the proof of Proposition \ref{proposition3}.
\end{proof}

\begin{proof}[Proof of Proposition \ref{proposition2}]
   To simplify the notations, from now on we denote by $C$ various constants which vary from place to place but do not depend on $\varepsilon$ and $r$. \\

By Lemma \ref{lem1} applied with $y:=y_1(t) - S(t) y_0^H$,  $h(t):=u_\varepsilon (t,y(t))$ and $\tilde h:= (y_1+y_2)(y_1+y_2)_x$ and by  Proposition \ref{proposition3}, we have
\begin{align}
\label{estimatey1-Sy0}
\lVert y_1(t) - S(t) y_0^H \lVert_{\mathcal{B}} & \leqslant  C\big(  \lVert u_{\varepsilon} \lVert_{L^2_T} + \lVert P_H ((y_1+y_2)(y_1+y_2)_x)\lVert_{L^1_TL^2_L}\big) \notag  \\
 & \leqslant  C\big(  \varepsilon\lVert y_2 \lVert_{L^1_TL^2_L}^{\frac{1}{2}} + \lVert y_1+y_2\lVert_{\mathcal{B}}^2\big) \notag  \\
 & \leqslant  C\big(  \varepsilon\lVert y_2 \lVert_{\mathcal{B}}^{\frac{1}{2}} + \lVert y_1\lVert_{\mathcal{B}}^2 +\lVert y_2\lVert_{\mathcal{B}}^2\big) \notag  \\
 & \leqslant  C( \varepsilon + \sqrt{r}) \big(   \lVert y_0^H \lVert_{L^2_L}^2  +\lVert y_0^M  \lVert_{L^2_L}   \big)^{\frac{1}{2}},
\end{align}
where $r := \lVert y_0 \lVert_{L^2_L} < r_{\varepsilon }<1$. On $r_{\varepsilon }$, we impose that
\begin{equation}
\label{repluspetite12}
r_{\varepsilon }  < \varepsilon^{12}.
\end{equation}
 From \eqref{estimatey1-Sy0} and \eqref{repluspetite12}, we have
\begin{align}\label{a}
\lVert y_1(t) - S(t) y_0^H \lVert_{\mathcal{B}}
 & \leqslant  C \varepsilon \big(   \lVert y_0^H \lVert_{L^2_L}^2  +\lVert y_0^M  \lVert_{L^2_L}   \big)^{\frac{1}{2}}.
\end{align}
Notice that, by Lemma \ref{lem1}, we have
\begin{align}
\lVert S(t) y_0^M \lVert_{\mathcal{B}} & \leqslant  C \lVert y_0^M \lVert_{L^2_L}, \\
\label{StH-stable}
\lVert S(t) y_0^H \lVert_{\mathcal{B}} & \leqslant  C \lVert y_0^H \lVert_{L^2_L}.
\end{align}
Proceeding as in   the proof of \eqref{a}, we have
\begin{align}\label{b}
\lVert y_2(t) - S(t) y_0^M \lVert_{\mathcal{B}} & \leqslant  C   \lVert P_M ((y_1+y_2)(y_1+y_2)_x)\lVert_{L^1_TL^2_L} \notag  \\
 & \leqslant  C\lVert y_1+y_2\lVert_{\mathcal{B}}^2 \notag  \\
 & \leqslant C \big(    \lVert y_2 \lVert_{\mathcal{B}} + \lVert S(t)y_0^H \lVert_{\mathcal{B}} +   \varepsilon \big(   \lVert y_0^H \lVert_{L^2_L}^2  +\lVert y_0^M  \lVert_{L^2_L}   \big)^{\frac{1}{2}}    \big)^2  \notag \\
 & \leqslant C \big(  (r + \varepsilon^2)(   \lVert y_0^H \lVert_{L^2_L}^2  +\lVert y_0^M  \lVert_{L^2_L}  ) +    \lVert y_0^H \lVert_{L^2_L}^2 \big)  \notag \\
 & \leqslant C\big(  \varepsilon^2 \lVert y_0^M  \lVert_{L^2_L} +    \lVert y_0^H \lVert_{L^2_L}^2 \big).
\end{align}
Let us now study successively the  two following cases

\begin{gather}
\label{H>M}
\lVert y_0^H \lVert_{L^2_L} \geqslant \varepsilon^{\frac{2}{3}} \sqrt{\lVert y_0^M \lVert_{L^2_L}},
\\
\label{H<M}
\lVert y_0^H \lVert_{L^2_L} < \varepsilon^{\frac{2}{3}} \sqrt{\lVert y_0^M \lVert_{L^2_L}}.
\end{gather}

We start with the case where \eqref{H>M} holds. From \hyperref[P1]{($\mathcal{P}_1$)},  \hyperref[P2]{($\mathcal{P}_2$)},  \eqref{a}, \eqref{b} and \eqref{H>M}, we get the existence of $\varepsilon_2 \in (0,\varepsilon_1)$ such that, for every $\varepsilon\in (0,\varepsilon_2)$,
\begin{align}
&\;\;\;\;\;\;\;\lVert y_1(T) \lVert^2_{L^2_L} +\varepsilon \lVert y_2(T) \lVert_{L^2_L}   \notag\\
&\leqslant \big(   C \varepsilon \big(   \lVert y_0^H \lVert_{L^2_L}^2  +\lVert y_0^M  \lVert_{L^2_L}   \big)^{\frac{1}{2}} +   \lVert S(T)y_0^H \lVert_{L^2_L} \big)^2   + \varepsilon \big( C\big(  \varepsilon^2 \lVert y_0^M  \lVert_{L^2_L} +    \lVert y_0^H \lVert_{L^2_L}^2 \big) +  \lVert S(T)y_0^M \lVert_{L^2_L} \big) \notag \\
&\leqslant (\rho_1 \rho_2)^{\frac{1}{2}} \lVert y_0^{H} \lVert^2_{L^2_L}+  C\varepsilon^2  \big(   \lVert y_0^H \lVert_{L^2_L}^2  +\lVert y_0^M  \lVert_{L^2_L}   \big)   + C\varepsilon   \lVert y_0^H \lVert_{L^2_L}^2 + (\varepsilon + C \varepsilon^3) \lVert y_0^M \lVert_{L^2_L}\notag\\
&\leqslant \rho_2 \lVert y_0^{H} \lVert^2_{L^2_L} +\varepsilon (1-\delta \varepsilon^2) \lVert y_0^{M} \lVert_{L^2_L}.
\end{align}

Let us now study the case where \eqref{H<M} holds. Let us define
\begin{equation}\label{def-b}
b:= y_0^M.
\end{equation}
 Then, from \eqref{a}, \eqref{StH-stable}, \eqref{b} and \eqref{H<M},  we get
\begin{align}\label{y1small}
\lVert y_1(t)  \lVert_{\mathcal{B}}
 & \leqslant  \lVert S(t) y_0^H \lVert_{\mathcal{B}}  + C \varepsilon \big(   \lVert y_0^H \lVert_{L^2_L}^2  +\lVert y_0^M  \lVert_{L^2_L}   \big)^{\frac{1}{2}} \notag \\
 & \leqslant  C \varepsilon \sqrt{\lVert b \lVert_{L^2_L}} + C \lVert y_0^H \lVert_{L^2_L}  \notag \\
 &\leqslant  C \varepsilon^{\frac{2}{3}} \sqrt{\lVert b \lVert_{L^2_L}},
\end{align}
and
\begin{equation}\label{close}
\lVert y_2(t) - S(t) y_0^M \lVert_{\mathcal{B}}  \leqslant \varepsilon^{\frac{4}{3}} \lVert b \lVert_{L^2_L},
\end{equation}
which shows that $y_2(\cdot)$ is close to $S(\cdot)y_0^M$. Let $z: [0,T]\to L^2(0,L)$
be the solution to the Cauchy problem
\begin{gather}\label{z1}
\begin{cases}
z_{1t}+z_{1xxx}+z_{1x}=0 \;\;\;& \textrm{in}\;\; (0,T) \times (0,L),\\
z_1(t,0)= z_1(t,L)=0  & \textrm{on}\;\; (0,T),\\
\displaystyle
z_{1x}(t,L)= v(t,\frac{b}{\lVert b \lVert_{L^2_L}})  & \textrm{on} \;\;(0,T),\\
z_1(0,x)= 0 & \textrm{on}\;\; (0,L).
\end{cases}
\end{gather}
 From \hyperref[P3]{($\mathcal{P}_3$)}, we know that $z_1(T) = 0$. Moreover, Lemma \ref{lem1} tells us that
\begin{equation}
\lVert z_1 (t) \lVert_{\mathcal{B}} \leqslant C \lVert v(t,\frac{b}{\lVert b \lVert_{L^2_L}}) \lVert_{L^2_T} \leqslant C.
\end{equation}
Let us define $w_1$ by
\begin{equation}\label{w11}
w_1 := y_1 - S(t) y_0^H - \varepsilon \lVert b \lVert_{L^2_L}^{\frac{1}{2}} z_1.
\end{equation}
 Then $w_1$ is the solution to the Cauchy problem
\begin{gather}\label{w1}
\begin{cases}
w_{1t}+w_{1xxx}+w_{1x} + P_H \big( (y_1 +y_2)(y_1 +y_2)_x \big)=0, \\
w_1(t,0)= w_1(t,L)=0, \\
\displaystyle
w_{1x}(t,L)= \varepsilon \left( \lVert y_2 (t)\lVert_{L^2_L}^{\frac{1}{2}} v(t, \frac{S(-t)y_2(t)}{\lVert y_2 (t)\lVert_{L^2_L}})  -   \lVert b \lVert_{L^2_L}^{\frac{1}{2}}  v(t,\frac{b}{\lVert b \lVert_{L^2_L}}) \right), \\
w_1(0,x)=0.
\end{cases}
\end{gather}
By Lemma \ref{lem1}, we get
\begin{align}\label{wes}
\lVert w_1 \lVert_{\mathcal{B}} \leqslant &C \lVert P_H \big( (y_1 +y_2)(y_1 +y_2)_x \big)\lVert_{L^1_TL^2_L} \notag\\
   &  + \varepsilon C \lVert \big( \lVert y_2 (t)\lVert_{L^2_L}^{\frac{1}{2}} v(t, \frac{S(-t)y_2(t)}{\lVert y_2 (t)\lVert_{L^2_L}})  -   \lVert b \lVert_{L^2_L}^{\frac{1}{2}}  v(t,\frac{b}{\lVert b \lVert_{L^2_L}}) \big) \lVert_{L^2_T}.
\end{align}
 Note that \eqref{close} insures that the right hand side of \eqref{wes} is of order $\varepsilon^2$. Indeed, for the first term  of the right-hand side of inequality \eqref{wes}, we have,
 using \eqref{repluspetite12}, \eqref{y1small} and \eqref{close},
\begin{align}\label{71}
C \lVert P_H \big( (y_1 +y_2)(y_1 +y_2)_x \big) \lVert_{L^1_TL^2_L}  &\leqslant  C \lVert y_1 + y_2 \lVert_{\mathcal{B}}^2  \notag\\
&\leqslant C \varepsilon^{\frac{4}{3}} \lVert b \lVert_{L^2_L} + C \lVert b \lVert_{L^2_L}  \notag \\
&\leqslant C  \lVert b \lVert_{L^2_L}^{\frac{1}{2}} \lVert b \lVert_{L^2_L}^{\frac{1}{2}} \notag\\
&\leqslant C  \varepsilon^6 \lVert b \lVert_{L^2_L}^{\frac{1}{2}}.
\end{align}
For the second term of the right hand side of inequality \eqref{wes},    by
\eqref{SrotationonM}, the Lipschitz condition \eqref{p3lip} on $v$ and \eqref{close}, we get, for every  $t \in [0,T]$,
\begin{align}\label{72}
& \;\;\;\mid  \lVert b \lVert_{L^2_L}^{\frac{1}{2}}  \big(   v(t,\frac{b}{\lVert b \lVert_{L^2_L}}) -  v(t, \frac{S(-t)y_2(t)}{\lVert y_2 (t)\lVert_{L^2_L}})\big)  \mid  \notag \\
& \leqslant  C \lVert b \lVert_{L^2_L}^{\frac{1}{2}}  \lVert \big( \frac{b}{\lVert b \lVert_{L^2_L}}-   \frac{S(-t)y_2(t)}{\lVert y_2 (t)\lVert_{L^2_L}}\big) \lVert_{L^2_L}  \notag \\
& \leqslant C \lVert b \lVert_{L^2_L}^{- \frac{1}{2}} \lVert y_2 (t)\lVert_{L^2_L}^{-1} \big( \lVert y_2 (t)\lVert_{L^2_L} \lVert b - S(-t) y_2(t) \lVert_{L^2_L} +  \lVert S(-t) y_2 (t)\lVert_{L^2_L} |  \lVert y_2(t) \lVert_{L^2_L} - \lVert b \lVert_{L^2_L}|  \big) \notag \\
& \leqslant C  \varepsilon^{\frac{4}{3}} \lVert b \lVert_{L^2_L}^{\frac{1}{2}},
\end{align}
and
\begin{align}\label{73}
 |  (\lVert y_2(t) \lVert_{L^2_L}^{\frac{1}{2}} -  \lVert b \lVert_{L^2_L}^{\frac{1}{2}} ) v(t, \frac{S(-t)y_2(t)}{\lVert y_2 (t)\lVert_{L^2_L}} )|    & \leqslant C \varepsilon^{\frac{4}{3}} \lVert b \lVert_{L^2_L}^{\frac{1}{2}}.
\end{align}
Combining \eqref{71}, \eqref{72} and \eqref{73}, we obtain the following estimate on $w_1$
\begin{equation}\label{esw1}
\lVert w_1 \lVert_{\mathcal{B}} \leqslant C \varepsilon^2 \lVert b \lVert_{L^2_L}^{\frac{1}{2}}.
\end{equation}
We fix
\begin{equation}\label{rho-comp}
\rho_3 \in  (\rho_1, \rho_2).
\end{equation}
 Then, by \eqref{w11}, \hyperref[P1]{($\mathcal{P}_1$)} as well as the fact that $z_1(T) = 0$, we get
\begin{equation}\label{y1t}
\lVert y_1(T) \lVert_{L^2_L}^2 \leqslant \rho_3 \lVert y_0^H \lVert_{L^2_L}^2 + C \varepsilon^4 \lVert b \lVert_{L^2_L}.
\end{equation}

We then come to the estimate of $y_2$. Let $\tau_1(t) := S(t) y_0^H$ and let $\tau_2:[0,T]\to L^2(0,L)$ and $z_2:[0,T]\to L^2(0,L) $ be the solutions to the Cauchy problems
\begin{gather}\label{tau2}
\begin{cases}
\tau_{2t}+\tau_{2xxx}+\tau_{2x} +  P_M (\tau_1 y_{1x} + \tau_{1x} y_1) -P_M (\tau_1 \tau_{1x})=0, \\
\tau_2(t,0)= \tau_2(t,L)=0, \\
\tau_{2x}(t,L)= 0, \\
\tau_2(0,x)= 0,
\end{cases}
\end{gather}
and
\begin{gather}\label{z2}
\begin{cases}
z_{2t}+z_{2xxx}+z_{2x} + P_M (z_1 z_{1x})=0,\;\;\;\;\;\;\;\;\;\;\;\; \\
z_2(t,0)= z_2(t,L)=0, \\
z_{2x}(t,L)= 0, \\
z_2(0,x)= 0.
\end{cases}
\end{gather}
Lemma~\ref{lem1}, Lemma~\ref{lem7}, \eqref{H<M} and \eqref{y1small} show us that
\begin{align}\label{7}
\lVert \tau_2 \lVert_{\mathcal{B}} & \leqslant C \lVert P_M \big( \tau_1 y_{1x} + \tau_{1x} y_1 -\tau_1 \tau_{1x} \big) \lVert_{L^1_TL^2_L} \notag \\
& \leqslant C \lVert \tau_1 \lVert_{\mathcal{B}} (\lVert y_1 \lVert_{\mathcal{B}} +\lVert \tau_1 \lVert_{\mathcal{B}} )   \notag \\
& \leqslant  C \varepsilon^{\frac{2}{3}} \lVert b \lVert_{L^2_L}^{\frac{1}{2}} \lVert y_0^H \lVert_{L^2_L},
\end{align}
and
\begin{equation}
\lVert z_2 \lVert_{\mathcal{B}} \leqslant \lVert z_1 \lVert_{\mathcal{B}}^2  \leqslant C.
\end{equation}
 From \hyperref[P3]{($\mathcal{P}_3$)}, \eqref{z1} and \eqref{z2},   we get
\begin{equation}
\big{<} z_2 (T), S(T) b \big{>}_{(L^2_L,L^2_L)}  \;<\; -2\delta \lVert b \lVert_{L^2_L}.
\end{equation}
Hence
\begin{align}\label{8}
\lVert S(T) b + \varepsilon^2 \lVert b \lVert_{L^2_L} z_2(T) \lVert_{L^2_L} &=  \Big ( \big{<} S(T) b + \varepsilon^2 \lVert b \lVert_{L^2_L} z_2(T), S(T) b + \varepsilon^2 \lVert b \lVert_{L^2_L} z_2(T)   \big{>}_{(L^2_L, L^2_L)} \Big )^{\frac{1}{2}}  \notag  \\
& \leqslant \Big ( \lVert b \lVert_{L^2_L}^2 +  \varepsilon^4 \lVert b \lVert_{L^2_L}^2 C  - 4 \delta \varepsilon^2 \lVert b \lVert_{L^2_L}^2  \Big )^{\frac{1}{2}}  \notag \\
& \leqslant  \lVert b \lVert_{L^2_L} (1- 2\delta \varepsilon^2 + C \varepsilon^4).
\end{align}
Let us define $w_2:[0,T]\to L^2(0,L)$ by
\begin{equation}\label{w2}
w_2:= y_2 -\tau_2 -\varepsilon^2 \lVert b \lVert_{L^2_L} z_2 -S(t)b.
\end{equation}
Then, from \eqref{skdv}, \eqref{tau2} and \eqref{z2}, we get that
\begin{align*}
w_{2t} & = y_{2t} - \tau_{2t} -\varepsilon^2 \lVert b \lVert_{L^2_L} z_{2t} - (S(t)b)_t   \notag \\
            &=  -w_{2x} -w_{2xxx} - P_M\big( (y_1 +y_2)(y_1 +y_2)_x \big) + P_M\big( \tau_1 y_{1x} + \tau_{1x}y_1\big)\\
            & \;\;\;\;\;\;\; -P_M\big( \tau_1 \tau_{1x}\big) +   \varepsilon^2 \lVert b \lVert_{L^2_L}  P_M\big(z_1 z_{1x} \big) \notag \\
            & = -w_{2x} -w_{2xxx} -  \varepsilon  \lVert b \lVert_{L^2_L}^{\frac{1}{2}} P_M\big( w_1z_{1x} + w_{1x}z_1 \big) - P_M \big( w_1 w_{1x} \big)\\
            & \;\;\;\;\;\;\; - P_M\big( y_1 y_{2x} + y_2 y_{1x} + y_2 y_{2x} \big).
\end{align*}
Hence, $w_2$ is the solution to the Cauchy problem
\begin{gather}
\begin{cases}
w_{2t}+w_{2xxx}+w_{2x} + \varepsilon  \lVert b \lVert_{L^2_L}^{\frac{1}{2}} P_M\big( w_1z_{1x} + w_{1x}z_1 \big) + P_M \big( w_1 w_{1x} \big)
\\
\phantom{bbbbbbbbbbbbbbbbbbbbbb} + P_M\big( y_1 y_{2x} + y_2 y_{1x} + y_2 y_{2x} \big)=0, \\
w_2(t,0)= w_2(t,L)=0, \\
w_{2x}(t,L)= 0, \\
w_2(0,x)= 0.
\end{cases}
\end{gather}
From Lemma \ref{lem1}, Lemma~\ref{lem7}, Proposition~\ref{proposition3}, \eqref{repluspetite12}, \eqref{H<M} and  \eqref{esw1},  we get
\begin{align}\label{esw2}
\lVert w_2 \lVert_{\mathcal{B}}  & \leqslant C\varepsilon \lVert b \lVert_{L^2_L}^{\frac{1}{2}} \lVert P_M\big( w_1z_{1x} + w_{1x}z_1 \big) \lVert_{L^1_TL^2_L} + C\lVert P_M\big( w_1 w_{1x} \big)  \lVert_{L^1_TL^2_L} \notag \\
 & \;\;\;\;\;\;\;\;\;\;\;\;+ C \lVert  P_M\big( y_1 y_{2x} + y_2 y_{1x} + y_2 y_{2x}\big)  \lVert_{L^1_TL^2_L}  \notag \\
& \leqslant C \varepsilon \lVert b \lVert_{L^2_L}^{\frac{1}{2}}  \varepsilon^2 \lVert b \lVert_{L^2_L}^{\frac{1}{2}}   +  C \varepsilon^4 \lVert b \lVert_{L^2_L} +  C \big(   \lVert y_0^H \lVert^2_{L^2_L} + \lVert y_0^M \lVert_{L^2_L}\big)^{\frac{3}{2}}   \notag \\
& \leqslant C \varepsilon^3 \lVert b \lVert_{L^2_L}.
\end{align}
We can now estimate $y_2(T)$ from \eqref{7}, \eqref{8}, \eqref{w2} and \eqref{esw2}:
\begin{align}\label{y2t}
\lVert y_2 (T) \lVert_{L^2_L} & = \lVert w_2(T) + \tau_2(T) + \varepsilon^2 \lVert b \lVert_{L^2_L} z_2(T) + S(T)b  \lVert_{L^2_L} \notag \\
& \leqslant \lVert b \lVert_{L^2_L} \big(   C \varepsilon^3 + 1- 2\delta \varepsilon^2 + C \varepsilon^4 \big) + C \varepsilon^{\frac{2}{3}} \lVert b \lVert_{L^2_L}^{\frac{1}{2}} \lVert y_0^H \lVert_{L^2_L}.
\end{align}
Combining\eqref{def-b}, \eqref{rho-comp}, \eqref{y1t} and \eqref{y2t}, we get existence of $\varepsilon_3>0 $ such that, for every $\varepsilon \in (0,\varepsilon_3]$, we have
\begin{equation}
\begin{array}{l}
\lVert y_1(T) \lVert^2_{L^2_L} +\varepsilon \lVert y_2(T) \lVert_{L^2_L}
\\
\phantom{bbbb}\leqslant \rho_3 \lVert y_0^H \lVert^2_{L^2_L} + C \varepsilon^4 \lVert b \lVert_{L^2_L} + \varepsilon \left(  \lVert b \lVert_{L^2_L} \big(   C \varepsilon^3 + 1- 2\delta \varepsilon^2 + C \varepsilon^4 \big) + C \varepsilon^{\frac{2}{3}} \lVert b \lVert_{L^2_L}^{\frac{1}{2}} \lVert y_0^H \lVert_{L^2_L} \right)
\\
\phantom{bbbb}\leqslant \rho_2 \lVert y_0^{H} \lVert^2_{L^2_L} +\varepsilon (1-\delta \varepsilon^2) \lVert y_0^{M} \lVert_{L^2_L}.
\end{array}
\end{equation}
This concludes the proof of Proposition \ref{proposition2}.
\end{proof}
\textbf{Acknowledgments.} We thank Jixun Chu, Ludovick Gagnon, Peipei Shang and Shuxia Tang for useful comments on a preliminary version of this article.

\appendix{}
\section{Appendix: Proof of Proposition \ref{proposition1}}
\label{sec-appendix-A}
\begin{proof}[Proof of Proposition \ref{proposition1}]
It is clear that, if $(y_1,y_2)$ is a solution to \eqref{skdv}, then $y$ is solution to \eqref{kdv-closed-loop-I}.  Let us assume that $y$ is a solution
 to the Cauchy problem \eqref{kdv-closed-loop-I}.  Then, by Definition \ref{definition3}, for every $\tau\in [s,T]$ and
for every $\phi \in C^3 ([s,\tau] \times [0,L])$ satisfying
\begin{equation}
\label{condition-bord-phi}
\phi(t,0) = \phi(t,L) = \phi_x (t,0) =0,\;\;\; \forall t \in [s,\tau],
\end{equation}
we have
\begin{align}\label{def0}
-\int_s^{\tau} \int_0^L &(\phi_t + \phi_x + \phi_{xxx})y dxdt - \int_s^{\tau} u(t,y(t,\cdot)) \phi_x (t,L) dt + \int_s^{\tau} \int_0^L \phi y y_x dxdt  \notag\\
 & + \int_0^L y(\tau,x)\phi(\tau,x) dx - \int_0^L y_0 \phi(s,x) dx = 0.
\end{align}

Let us denote by $\phi_1$ and $\phi_2$  the projection of $\phi$ on $H$ and $M$ respectively: $\phi_1:=P_H(\phi)$, $\phi_2:=P_M(\phi)$. Because $M$ is spanned by $\varphi_1^j$ and $ \varphi_2^j$, $j\in \{1,\ldots,n\}$, which are of class $C^\infty$ and satisfy
\begin{equation*}
\varphi_1^j (0) = \varphi_1^j (L) = \varphi_{1x}^j (0) = \varphi_{1x}^j (L) = 0,
\end{equation*}
\begin{equation*}
\varphi_2^j (0) = \varphi_2^j (L) = \varphi_{2x}^j (0) = \varphi_{2x}^j (L) = 0,
\end{equation*}
the functions $\phi_1, \phi_2 \in C^3 ([s,\tau] \times [0,L])$ and satisfy
\begin{gather}
\label{condition-bord-phi-1}
{\color{black} \phi_1(t,0) = \phi_1(t,L)= \phi_{1x} (t,0) =0, \forall t \in [s,\tau],}
\\
\label{condition-bord-phi-2}
\phi_2(t,0) = \phi_2(t,L) = \phi_{2x} (t,0) = \phi_{2x}(t,L) = 0, \forall t \in [s,\tau].
\end{gather}
Using \eqref{def0} for $\phi=\phi_2$ in \eqref{def0} together with \eqref{condition-bord-phi-2}, we get
\begin{align}
\label{premieres-integrale-phi2}
-\int_s^{\tau} \int_0^L &(\phi_{2t} + \phi_{2x} + \phi_{2xxx})y dxdt  + \int_s^{\tau} \int_0^L \phi_2 y y_x dxdt  \notag\\
 & + \int_0^L y(\tau,x)\phi_2(\tau,x) dx - \int_0^L y_0 \phi_2(s,x) dx = 0,
\end{align}
which, combined with the fact that $\phi_{2t} + \phi_{2x} + \phi_{2xxx}\in M$, gives
\begin{align}\label{def2}
-\int_s^{\tau} \int_0^L &(\phi_{2t} + \phi_{2x} + \phi_{2xxx})y_2 dxdt  + \int_s^{\tau} \int_0^L \phi_2 P_M (y y_x) dxdt  \notag\\
 & + \int_0^L y_2 (\tau,x)\phi_2(\tau,x) dx - \int_0^L P_M(y_0)\phi_2(s,x) dx = 0.
 \end{align}
 Simple integrations by parts show that $\phi_{1x} + \phi_{1xxx} \in M^\bot =H$. Since, $\phi_1$ and $\phi_{1t}$ are also in $H$, we get from \eqref{def2} that
 \begin{align}\label{def3}
-\int_s^{\tau} \int_0^L &(\phi_{t} + \phi_{x} + \phi_{xxx})y_2 dxdt  + \int_s^{\tau} \int_0^L \phi P_M (y y_x) dxdt  \notag\\
 & + \int_0^L y_2 (\tau,x)\phi(\tau,x) dx - { \color{black} \int_0^L P_M(y_0) \phi(s,x) dx} = 0,
 \end{align}
 which is exactly the definition of a solution of the second part of the linear KdV system \eqref{skdv}.
 We then combine \eqref{def0} and \eqref{def3} to get
 \begin{align}
-\int_s^{\tau} \int_0^L &(\phi_t + \phi_x + \phi_{xxx})y_1 dxdt - \int_s^{\tau} u(t,y(t,\cdot)) \phi_x (t,L) dt + \int_s^{\tau} \int_0^L \phi P_H (y y_x) dxdt  \notag\\
 & + \int_0^L y_1 (\tau,x)\phi(\tau,x) dx - \int_0^L P_H(y_0) \phi(0,x) dx = 0,
\end{align}
 and we get the definition of a solution to the first part of the linear KdV system \eqref{skdv}. This concludes the
  proof of Proposition~\ref{proposition1}.
 \end{proof}

\section{Proofs of Theorem~\ref{th-feedback-Lipschitz} and Theorem \ref{th-feedback-not-Lipschitz}}
\label{appendix-proof-th-feedback-Lipschitz}
 Our strategy to prove Theorem~\ref{th-feedback-Lipschitz} is to prove first the existence of  a solution for small times and then to use some a priori estimates to control the $L^2_L$-norm of the solution with which we can {\color{black} extend} the solution to a longer time, and to  continue until the solution blows up.  We start by proving the following lemma.

\begin{lemma}\label{lem-small}
Let $C_2>0$ be as in Lemma~\ref{lem1} for $T_2-T_1=1$. Assume that $u$ is a Carath\'{e}odory function and  that, for every $R>0$, there exists $K(R)>0$ such that
\begin{equation}
\label{estimate-Lipschitz-K}
\left(\lVert y\rVert_{L^2_L}\leqslant R \text{ and } \lVert z\rVert_{L^2_L}\leqslant R\right)
\Rightarrow
\left(| u(t,y)-u(t,z)|\leqslant K(R) \lVert y- z\rVert_{L^2_L},\;\;\; \forall t\in \R\right).
\end{equation}
Then, for every $R \in (0, + \infty)$, there exists a time $T(R)>0$ such that, for every $s\in \R$ and for every $y_0\in L^2(0,L)$ with $\lVert y_0 \lVert_{L^2_L} \leqslant R$, the Cauchy problem \eqref{kdv-closed-loop-I} has one and only one solution $y$ on $[s, s+ T(R)]$.
Moreover, this solution satisfies
\begin{equation}\label{smallsolution}
\lVert y \lVert_{\mathcal{B}_{s, s + T(R)}} \leqslant C_R := 3C_2 R.
\end{equation}
\end{lemma}
\begin{proof}[Proof of Lemma \ref{lem-small}] Let us first point out that it follows from our choice of $C_2$ and  Lemma~\ref{lem1} that, for every $-\infty<T_1<T_2<+\infty$ such that
 $T_2 - T_1 \leqslant 1$, for every solution $y$ of problem \eqref{likdv},
 estimation \eqref{l1} holds.

 Let $y_0\in L^2(0,L)$  be such that
\begin{equation}\label{y0petitR}
\lVert y_0 \lVert_{L^2_L} \leqslant R.
\end{equation}
Let us  define $\mathcal{B}_1$ by
\begin{equation*}
\mathcal{B}_1 :=  \big\{  y \in \mathcal{B}_{s, s+ T(R)}  ;\,  \lVert y \lVert_{\mathcal{B}_{s, s+T(R)}} \leqslant C_R \big \}.
\end{equation*}
 The set $\mathcal{B}_1$ is a closed subset of $\mathcal{B}_{s, s+ T(R)}$. For every  $y \in \mathcal{B}_1$,
 we define $\Psi(y)$ as the solution of \eqref{likdv} with $\tilde{h} := - y y_x$, $h(t):= u(t, y(t,\cdot))$ and $y_0 := y_0$. Let us prove that, for $T(R)$ small enough, the smallness being independent of $y_0$ provided that it satisfies \eqref{y0petitR}, we have
\begin{equation}\label{maps}
\Psi(\mathcal{B}_1) \subset  \mathcal{B}_1.
\end{equation}
Indeed for $y \in \mathcal{B}_1$, by Lemma \ref{lem1} and Lemma~\ref{lem7}, we have, if $T(R)\leqslant 1$,
\begin{align}
\label{B1conserved-step1}
\lVert \Psi(y) \lVert_{\mathcal{B}} & \leqslant C_2\big( \lVert y_0 \lVert_{L^2_L} + \lVert h \lVert_{L^2_T} + \lVert \tilde{h} \lVert_{L^1 (0, T; L^2 (0,L))} \big) \notag \\
& \leqslant  C_2\big( \lVert y_0 \lVert_{L^2_L} + \lVert  u(t, y(t, \cdot)) \lVert_{L^2_T} + \lVert -yy_x \lVert_{L^1 (s, s+ T(R); L^2 (0,L))} \big) \notag \\
& \leqslant C_2 \big( R  + C_B(C_R) T(R)^{\frac{1}{2}}  + c_4 T(R)^{\frac{1}{4}} \lVert y \lVert_{\mathcal{B}}^{2} \big).
\end{align}
In \eqref{B1conserved-step1} and until the end of the proof of Lemma~\ref{lem-small}, for ease of notation, we simply write $\lVert \cdot \lVert_{\mathcal{B}}$ for
$\lVert \cdot \lVert_{\mathcal{B}_{s, s+ T(R)}}$. From \eqref{B1conserved-step1}, we get that, if
\begin{equation}\label{TR1}
T(R) \leqslant \textrm{min} \left\{ \big( \frac{R}{C_B(C_R)} \big)^2, \big( \frac{1}{9c_4 C_2^2 R} \big)^4 ,1\right\},
\end{equation}
then \eqref{maps} holds. From now on, we assume that \eqref{TR1} holds.

Note that every $y \in \mathcal{B}_1 $ such that $\Psi (y)  = y $ is a solution of \eqref{kdv-closed-loop-I}.  In order to use the Banach fixed point theorem, it remains to estimate $\lVert \Psi(y)-\Psi(z) \lVert_{\mathcal{B}}$.   We know that $\Psi(y)-\Psi(z)$ is the solution of equation  \eqref{likdv} with $T_1:=s$, $T_2=s+T(R)$, $\tilde{h} := - y y_x + z z_x$, $h(t) := u(t, y(t, \cdot))- u(t, z(t, \cdot)) $ and $y_0 := 0$.  Hence, from Lemma \ref{lem1}, Lemma \ref{lem7} and \eqref{estimate-Lipschitz-K}, we get that
\begin{align}
\lVert \Psi(y)-\Psi(z) \lVert_{\mathcal{B}} & \leqslant C_2\big( \lVert y_0 \lVert_{L^2_L} + \lVert h \lVert_{L^2_T} + \lVert \tilde{h} \lVert_{L^1 (0, T; L^2 (0,L))} \big) \notag \\
& \leqslant C_2 \big( 0+ T(R)^{\frac{1}{2}} K(C_R) \lVert y-z \lVert_{\mathcal{B}} + c_4 T(R)^{\frac{1}{4}}   \lVert y-z \lVert_{\mathcal{B}} ( \lVert y \lVert_{\mathcal{B}}+  \lVert z \lVert_{\mathcal{B}}) \big) \notag \\
& \leqslant C_2 \lVert y-z \lVert_{\mathcal{B}} \big( T(R)^{\frac{1}{2}} K(C_R)+  2 c_4 T(R)^{\frac{1}{4}} C_R \big),   \notag
\end{align}
which shows that, if
\begin{equation}\label{TR2}
T(R) \leqslant \textrm{min} \left\{ \big( \frac{1}{12c_4 C_2^2 R} \big)^4, \big( \frac{1}{4C_2 K(3C_2R)} \big)^2 \right\},
\end{equation}
then,
\begin{equation}
\lVert \Psi(y)-\Psi(z) \lVert_{\mathcal{B}}  \leqslant \frac{3}{4} \lVert y-z \lVert_{\mathcal{B}}.   \notag \\
\end{equation}

Hence, by the Banach fixed point theorem,  there exists $y \in \mathcal{B}_1 $ such that $\Psi (y)  = y $, which is {\color{black} the}  solution that we are looking for.  We define $T(R)$ as
\begin{equation}\label{TR}
T(R) := \textrm{min} \left\{ \big( \frac{R}{C_B(3C_2 R)} \big)^2, \big( \frac{1}{12c_4 C_2^2 R} \big)^4, \big( \frac{1}{4C_2 K(3C_2R)} \big)^2, 1 \right\}.
\end{equation}

It only remains to prove the uniqueness of the solution to the Cauchy problem \eqref{kdv-closed-loop-I} (the above proof gives only the uniqueness in the set  $\mathcal{B}_1$). Clearly it suffices to prove that two solutions  to \eqref{kdv-closed-loop-I-sans-y0} which are equal at a time $\tau$ are equal in a neighborhood of $\tau$ in $[\tau,+\infty)$. This property follows from the
above proof and from the fact that, for every
solution $y:[\tau,\tau_1]\to L^2(0,L)$ of \eqref{kdv-closed-loop-I}, then, if $T>0$ is small enough (the smallness depending on $y$),
\begin{equation}\label{dansB1pourTpetit}
  \lVert y \lVert_{\mathcal{B}_{\tau,\tau+T}} \leqslant 3C_2\lVert y(\tau) \lVert_{L^2_L}.
\end{equation}
\end{proof}

 Proceeding similarly as in the above  proof of Lemma \ref{lem-small},  one can get the following lemma concerning the Cauchy problem \eqref{nonlikdv}.
\begin{lemma}\label{lemma-small-H}
Let $C_2>0$ be as in Lemma~\ref{lem1} for $T_2-T_1=1$. Given $R, M > 0$, there exists $T(R, M) > 0$ such that,  for every $s \in \mathbb{R}$, for every $y_0 \in L^2(0,L)$ with $\lVert y_0 \lVert_{L^2_L} \leqslant R$, and for every measurable $H:(s, s + T(R, M))\to \R$ such that  $| H(t) | \leqslant M$ for every  $t\in (s, s + T(R, M))$,  the Cauchy problem
\begin{gather}\label{nonli-bound}
\begin{cases}
y_t+ y_{xxx}+ y_x+ y y_x= 0   \;\;\; & \textrm{in} \;\;(s, s+ T(R, M)) \times (0, L),\\
y(t,0)= y(t,L)= 0  & \textrm{on} \;\;(s, s+ T(R, M)), \\
y_x(t,L)= H(t) & \textrm{on} \;\;(s, s+ T(R, M)), \\
y(s,x)= y_0(x) & \textrm{on} \;\;(0, L),
\end{cases}
\end{gather}
has one and only one solution $y$ on $[s, s+ T(R,M)]$. Moreover, this solution satisfies
\begin{equation}
\lVert y \lVert_{\mathcal{B}_{s, s+ T(R, M)}} \leqslant 3C_2 R.
\end{equation}
\end{lemma}

We are now in position to prove Theorem ~\ref{th-feedback-Lipschitz}.
\begin{proof}[Proof of Theorem ~\ref{th-feedback-Lipschitz}]
The uniqueness follows from the proof of the uniqueness part of Lemma~\ref{lem-small}. Let us give the proof of the existence. Let $y_0 \in L^2(0, L)$, let $s \in \mathbb{R}$ and let $T_0:=T(\lVert y_0 \lVert_{L^2_L})$ . By Lemma \ref{lem-small},  there exists a solution
$y \in \mathcal{B}_{s, s+ T_0}$
to the Cauchy problem \eqref{kdv-closed-loop-I}. Hence, together with the uniqueness of the solution, we can find a maximal solution $y:D(y)\to L^2(0,L)$ with $[s, s + T_0] \subset D(y)$. By the maximality of the solution $y$
and Lemma~\ref{lem-small}, there exists $\tau\in [s+T_0,+\infty)$ such that $D(y) = [s, \tau)$.  Let us assume that $\tau<+\infty$ and that \eqref{divergenceattau} does not hold.  Then there exists an increasing sequence $(t_n)_{n\in \mathbb{N}}$
of real numbers in $(s,\tau)$ and $R\in (0,+\infty)$ such that
\begin{gather}
\label{cvtnverstau}
\lim_{n\rightarrow +\infty}t_n=\tau,
\\
\label{convergenceattau}
\lVert y(t_n)\rVert_{L^2_L}\leq R,  \;\;\forall n \in \mathbb{N}.
\end{gather}
By \eqref{cvtnverstau}, there exists $n_0\in \mathbb{N}$ such that
\begin{gather}\label{n0grandR}
t_{n_0}\geq \tau-T(R)/2.
\end{gather}
 From Lemma~\ref{lem-small}, there is a solution $z: [t_{n_0},t_{n_0}+T(R)]\to L^2(0,L)$ of \eqref{kdv-closed-loop-I} for the initial time $s:=t_{n_0}$
 and the initial data $z(t_{n_0}):=y(t_{n_0})$. Let us then define $\tilde y :[s,t_{n_0}+T(R)]\to L^2(0,L)$ by
\begin{gather}\label{deftildeytpetit}
  \tilde y (t):=y(t),  \;\;\forall t \in [s,t_{n_0}], \\
  \tilde y (t):=z(t),  \;\;\forall t \in [t_{n_0},t_{n_0}+T(R)].
\end{gather}
Then $\tilde y$ is also a solution to the Cauchy problem \eqref{kdv-closed-loop-I}. By the uniqueness of this solution, we have
$y=\tilde y$ on $D(y)\cap D(\tilde y)$. However, from \eqref{n0grandR},  we have that $D(y) \subsetneqq D(\tilde y)$, in contradiction with the maximality of $y$.

Finally, we prove that, if $C(R)$ satisfies \eqref{Clinearlybounder}, then, for the maximal solution $y$ to \eqref{kdv-closed-loop-I}, we have $D(y) = [s, + \infty)$. We argue by contradiction and therefore assume that the maximal solution $y$ is
such that $D(y) = [s, \tau)$  with $\tau<+\infty$. Then \eqref{divergenceattau} holds. Let us estimate
$\lVert y(t) \lVert_{L^2_L}$ when $t$ tends to $\tau^{-}$.  We define the energy $E:[s,\tau)\to [0,+\infty)$ by
\begin{equation}\label{energy}
E(t) :=  \int_0^L | y(t, x) |^2 dx.
\end{equation}
Then $E\in C^0([s,\tau))$ and, in the distribution sense, it satisfies
\begin{equation}
\label{decay-energy-E-C}
\frac{d E}{d t} \leqslant  | u(t, y(t, \cdot)) |^2 \leqslant  C^2_{{\color{black} B}}(\sqrt{E}).
\end{equation}
(We get such estimate first in the classical sense for regular initial data and regular boundary conditions $y_x(t,L)=\varphi(t)$ with the related compatibility conditions; the general case then follows
from this special case by smoothing the initial data and the boundary conditions, by passing to the limit, and by using the uniqueness of the solution.)
From \eqref{divergenceattau} and  \eqref{decay-energy-E-C}, we get that
\begin{equation}\label{divpourE}
\frac{1}{2}\int_0^{+\infty } \frac{1}{C^2_{{\color{black} B}}(\sqrt{E})}dE <+\infty.
\end{equation}
However the left hand side of \eqref{divpourE} is equal to the left hand side of \eqref{Clinearlybounder}. Hence
 \eqref{Clinearlybounder} and \eqref{divpourE} are in contradiction.
This completes the proof of Theorem \ref{th-feedback-Lipschitz}.
\end{proof}

The proof of Theorem \ref{th-feedback-not-Lipschitz} is more difficult. For this proof,  we adapt a strategy introduced by Carath\'{e}odory to solve ordinary differential equations $\dot y=f(t,y)$ when $f$ is not smooth. Roughly speaking it consists in solving
$\dot y =f(t, y(t-h))$ where $h$ is a positive time-delay  and then let $h$ tend to $0$. Here we do not put the time-delay on $y$ (it does not seem to be possible) but only on the feedback law: $ u(t,y(t))$ is replaced by $ u(t,y(t-h))$.

\begin{proof}[Proof of Theorem \ref{th-feedback-not-Lipschitz}]
Let us  define $H: [0,+\infty) \rightarrow [0,+\infty)$ by
 \begin{equation}
 H(a) := \int_0^{a} \frac{1}{(C_B(\sqrt{E}))^2} dE=2\int_0^{{\color{black} \sqrt{a}}}\frac{R}{(C_B(R))^2} dR.
 \end{equation}
 From \eqref{Clinearlybounder}, we know that $H$ is a bijection from $[0,+\infty)$ into $[0,+\infty)$. We denote by $H^{-1}:[0,+\infty)\to [0,+\infty)$ the inverse of this map.

For  given $y_0 \in L^2(0, L)$ and $s\in \mathbb{R}$, let us prove that there exists a solution $y$ defined on $[s, + \infty)$ to the Cauchy problem \eqref{kdv-closed-loop-I}, which also satisfies
\begin{equation}\label{yes}
\lVert y(t) \lVert_{L^2(0, L)}^2 \leqslant  H^{-1}\left(H \left(\lVert y(s) \lVert_{L^2_L}^2\right) + (t - s)\right)  < + \infty,\;\;
\forall  t \in [s, + \infty).
\end{equation}

Let $n \in \mathbb{N}^{*}$. Let us  consider the following Cauchy system on $[s, s+ 1/n]$
\begin{gather}
\begin{cases}\label{sys0}
y_t+y_{xxx}+y_x+y y_x=0 \;\;\;\;\;&\textrm{in}\;\;  (s, s + (1/n)) \times (0,L), \\
y(t,0)= y(t,L)=0   &\textrm{on}\;\; (s, s +  (1/n)),\\
y_x(t,L)=u(t,y_0)    &\textrm{on}\;\; (s, s + (1/n)), \\
y(s,x)=y_0(x) &\textrm{on}\;\;  (0,L).
\end{cases}
\end{gather}
By Theorem~\ref{th-feedback-Lipschitz} applied with the feedback law $(t,y)\mapsto u(t,y_0)$
(a measurable bounded feedback law which now does not depend on $y$ and therefore satisfies \eqref{lip-condition-R}),
the Cauchy problem \eqref{sys0} has  one and only one solution $y$. Let us  now  consider the following Cauchy problem on $[s+(1/n), s+ (2/n)]$
\begin{gather}
\begin{cases}\label{sys-1-sur-n}
y_t+y_{xxx}+y_x+y y_x=0 \;\;\;\;\;&\textrm{in}\;\;  (s+(1/n), s+ (2/n)) \times (0,L), \\
y(t,0)= y(t,L)=0   &\textrm{on}\;\; (s+(1/n), s+ (2/n)),\\
y_x(t,L)=u(t,y(t-(1/n)))    &\textrm{on}\;\; (s+(1/n), s+ (2/n)), \\
y(s,x)=y_0(x) &\textrm{on}\;\;  (0,L).
\end{cases}
\end{gather}
As for \eqref{sys0}, this Cauchy problem  has one and only one solution, that we still denote by $y$. We keep going and, by induction
on the integer $i$, define $y\in C^0([s,+\infty);L^2(0,L))$ so that,
on $[s+(i/n), s+((i+1)/n)]$, $i\in \mathbb{N}\setminus\{0\}$, $y$ is the solution to the Cauchy problem
\begin{gather}
\begin{cases}\label{sys-i-sur-n}
y_t+y_{xxx}+y_x+y y_x=0 \;\;\;\;\;&\textrm{in}\;\;  (s+(i/n), s+((i+1)/n)) \times (0,L), \\
y(t,0)= y(t,L)=0   &\textrm{on}\;\; (s+(i/n), s+((i+1)/n)),\\
y_x(t,L)=u(t,y(t-(1/n)))    &\textrm{on}\;\; (s+(i/n), s+((i+1)/n)), \\
y(s+(i/n))=y(s+(i/n)-0) &\textrm{on}\;\;  (0,L),
\end{cases}
\end{gather}
where, in the last equation, we mean that the initial value, i.e. the value at time $(s+(i/n))$, is the value at time $(s+(i/n))$ of the $y$
defined previously on $[(s+((i-1)/n)), s+(i/n)]$.

Again, we let, for $t\in [s,+\infty)$,
\begin{equation}\label{energy-again}
E(t) :=  \int_0^L | y(t, x) |^2 dx.
\end{equation}
Then $E\in C^0([s,+\infty))$ and, in the distribution sense, it satisfies (compare with \eqref{decay-energy-E-C})
\begin{gather}
\label{decay-energy-E-C-0}
\frac{d E}{d t} \leqslant  | u(t, y_0) |^2 \leqslant  C^2_B(\sqrt{E(s)}), \, t\in (s, s+ (1/n)),
\\
\label{decay-energy-E-C-i}
\frac{d E}{d t} \leqslant  | u(t, y(t-(1/n)) |^2 \leqslant  C^2_B(\sqrt{E(t-(1/n))}), \, t\in (s+(i/n), s+((i+1)/n)), i>0.
\end{gather}
Let $\varphi :[0,+\infty)\to [0,+\infty)$ be the solution of
\begin{equation}\label{defvarphi}
\frac{d \varphi }{d t}= C_B^2\left(\sqrt{\varphi(t)}\right),\, \varphi(s)= E(s).
\end{equation}
Using \eqref{decay-energy-E-C-0}, \eqref{decay-energy-E-C-i}, \eqref{defvarphi} and simple comparaison arguments, one gets
 that
\begin{equation}\label{comparaison-E-varphi}
E(t)\leqslant \varphi (t), \;\; \forall t\in[s,+\infty),
\end{equation}
 i.e.
\begin{equation}
\label{decay-energy-E-C-final}
E(t)\leqslant H^{-1}\left(H\left(E(s)\right)+ (t-s)\right),\;\;\forall t\in [s, +\infty).
\end{equation}

We now want to let $n\rightarrow +\infty$. In order to show the dependance on $n$, we write $y^n$ instead of $y$.
In particular \eqref{decay-energy-E-C-final} becomes
\begin{equation}
\label{decay-energy-E-C-i-avec-n}
\lVert y^n(t) \lVert_{L^2(0, L)}^2 \leqslant  H^{-1}\left(H (\lVert y_0 (s) \lVert_{L^2_L}^2) + (t - s)\right)
,\;\;\forall t\in [s, +\infty).
\end{equation}
 From Lemma~\ref{lemma-small-H}, \eqref{decay-energy-E-C-i-avec-n} and the construction of $y^n$, we get that, for every $T>s$,
 there exists $M(T)>0$ such that
\begin{equation}
\label{estimate-y-n-B}
 \lVert y^n \lVert_{\mathcal{B}_{s,  T}} \leqslant M(T), \;\; \forall n\in \mathbb{N}.
\end{equation}
Hence, upon extracting a subsequence of $(y^n)_{n}$ that we still denote by $(y^n)_{n}$, there exists

\begin{equation}\label{property-y-limit}
y \in L^\infty_{\text{loc}}([s,+\infty);L^2(0,L))\cap L^2_{\text{loc}}([s,+\infty);H^1(0,L)),
\end{equation}
such that, for every $T>s$,
\begin{gather}
\label{weakLinftyL2}
y^n\rightharpoonup y \text{ in } L^\infty(s,T;L^2(0,L)) \text{ weak } * \text{ as } n\rightarrow +\infty,
\\
\label{weakL2H1}
y^n\rightharpoonup y \text{ in } L^2(s,T;H^1(0,L)) \text{ weak }  \text{ as } n\rightarrow +\infty.
\end{gather}
Let us define $z^n:[s,s+\infty)\times(0,L) \to \R$ and $\gamma^n:[s,+\infty)\to \R$ by
\begin{gather}
\label{defz-n-t-petit}
z^n(t):=y_0, \;\; \forall t \in [s,s+(1/n)],
\\
\label{defz-n-t-grand}
z^n(t):=y^n(t-(1/n)), \;\; \forall t\in (s+(1/n), +\infty),
\\
\label{def-gamma-n}
\gamma^n(t):=u(t,z^n), \;\; \forall t\in [s, +\infty).
\end{gather}
Note that $y^n$ is the solution to the Cauchy problem
\begin{gather}
\begin{cases}\label{sysyn}
y^n_t+y^n_{xxx}+y^n_x+y^n y^n_x=0 \;\;\;\;\;&\textrm{in}\;\;  (s,+\infty) \times (0,L), \\
y^n(t,0)= y^n(t,L)=0   &\textrm{on}\;\; (s,+\infty),\\
y^n_x(t,L)=\gamma^n(t)    &\textrm{on}\;\; (s,+\infty), \\
y^n(s,x)=y_0(x) &\textrm{on}\;\; (0,L).
\end{cases}
\end{gather}
 From \eqref{estimate-y-n-B} and the first line of \eqref{sysyn}, we get that
\begin{equation}\label{esap1}
\forall T>0, \, \left( \frac{d}{d t} y^n \right)_{n\in \mathbb{N}}\;\; \textrm{ is bounded in } \;\;  L^2(s, s+T ; H^{-2}(0,L)).
\end{equation}
 From \eqref{weakLinftyL2}, \eqref{weakL2H1}, \eqref{esap1} and the Aubin-Lions Lemma \cite{1963-Aubin-CRAS}, we get that
 \begin{equation}\label{strong}
 y^n \rightarrow y \;\; \textrm{in} \;\; L^2(s,T; L^2(0,L))
\text{ as } n\rightarrow +\infty,\;\;\forall T >s.
\end{equation}
From \eqref{strong} we know that, upon extracting a subsequence if necessary, a subsequence still denoted by $(y^n)_n$,
\begin{equation}\label{yny}
\lim_{n\rightarrow +\infty} \lVert y^n (t) -y(t) \lVert_{L^2_L}=0, \text{ for almost every } t \in (s,+\infty).
\end{equation}
 Letting $n\rightarrow +\infty$ in inequality \eqref{decay-energy-E-C-final} for $y^n$ and using \eqref{yny}, we get that
 \begin{equation}\label{ineq-y-OK}
\lVert y(t) \lVert_{L^2(0, L)}^2 \leqslant  H^{-1}\left(H (\lVert y_0 \lVert_{L^2_L}^2) + (t - s)\right),
\text{ for almost every } t\in (0,+\infty).
\end{equation}
 Note that, for every $T>s$,
\begin{equation}\label{est-intermediaire-z-n-y}
\begin{array}{rcl}
\displaystyle
\lVert z^n -y \lVert_{L^2((s,T);L^2_L)}&\leq&
(1/\sqrt{n})\lVert y_0 \lVert_{L^2_L}
+
\lVert y^n(\cdot-(1/n))-y(\cdot-(1/n))\lVert_{L^2(s+(1/n),T;L^2(0,L))}
\\
&&
\displaystyle
+
\lVert y(\cdot-(1/n))-y(\cdot)\lVert_{L^2(s+(1/n),T;L^2(0,L))}
+
\lVert y \lVert_{L^2(s,s+(1/n);L^2(0,L))}
\\
&\leq&
\displaystyle
(1/\sqrt{n})\lVert y_0 \lVert_{L^2_L}
+
\lVert y^n-y\lVert_{L^2(s,T;L^2(0,L))}
\\
&&
\displaystyle
+
\lVert y(\cdot-(1/n))-y(\cdot)\lVert_{L^2(s+(1/n),T;L^2(0,L))}
+
\lVert y(\cdot)\lVert_{L^2(s,s+(1/n);L^2(0,L))}.
\end{array}
\end{equation}
From \eqref{defz-n-t-petit}, \eqref{defz-n-t-grand}, \eqref{strong} and \eqref{est-intermediaire-z-n-y}, we get that
\begin{equation}\label{strongzn}
z^n \rightarrow y \;\; \textrm{in} \;\; L^2(s,T; L^2(0,L))
\text{ as } n\rightarrow +\infty,\;\;\forall T >s.
\end{equation}
 Extracting, if necessary, from the sequence $ (z^n)_n$ a subsequence, a subsequence still denoted $(z^n)_n$,
 and using \eqref{strongzn}, we have
\begin{equation}\label{gamma}
\lim_{n\rightarrow +\infty} \lVert z^n (t) -y(t) \lVert_{L^2_L} = 0, \text{ for almost every } t\in (s,+\infty).
\end{equation}
 From \eqref{locallybounded}, \eqref{measurable}, \eqref{almostcontinuity}, \eqref{estimate-y-n-B}, \eqref{defz-n-t-petit}, \eqref{defz-n-t-grand} and \eqref{gamma},
extracting a subsequence  from the sequence $(\gamma ^n)_n$ if necessary, a subsequence still denoted $(\gamma ^n)_n$, we may assume that
\begin{equation}\label{cvgamman}
\gamma^n \rightharpoonup \gamma(t):=u(t,y(t)) \text{ in } L^\infty(s,T) \text{ weak } *
\text{ as } n\rightarrow +\infty,\;\;\forall T>s.
\end{equation}
Let us now check that
\begin{equation}\label{eq-y-solution}
  \text{$y$ is a solution to the Cauchy problem \eqref{kdv-closed-loop-I}. }
\end{equation}
Let $\tau\in [s,+\infty)$ and let
$\phi \in C^3([s, \tau] \times [0,L])$ be such that
\begin{equation}
\phi (t,0) = \phi (t,L) = \phi_x (t,0) = 0, \;\;\forall t \in [T_1,\tau].
\end{equation}
 From \eqref{sysyn}, one has, for every $n\in \mathbb{N}$,
\begin{align}
\label{y-n-solution-integral}
-\int_{T_1}^{\tau} \int_0^L &(\phi_t + \phi_x + \phi_{xxx})y^n dxdt - \int_{T_1}^{\tau} \gamma^n \phi_x (t,L) dt + \int_{T_1}^{\tau} \int_0^L \phi y^ny^n_x dxdt  \notag\\
 & + \int_0^L y(\tau,x)\phi(\tau,x) dx - \int_0^L y_0 \phi(s,x) dx = 0.
\end{align}
Let $\tau $ be such that
\begin{equation}\label{ynytau}
\lim_{n\rightarrow +\infty} \lVert y^n (\tau) -y(\tau) \lVert_{L^2_L}=0.
\end{equation}
Let us recall that, by \eqref{yny}, \eqref{ynytau} holds for almost every $\tau\in [s,+\infty)$. Using \eqref{weakL2H1}, \eqref{strong},
\eqref{cvgamman}, \eqref{ynytau} and letting $n\rightarrow +\infty $ in \eqref{y-n-solution-integral}, we get
\begin{align}
\label{y-solution-integral}
-\int_{T_1}^{\tau} \int_0^L &(\phi_t + \phi_x + \phi_{xxx})y dxdt - \int_{T_1}^{\tau} u(t,y(t)) \phi_x (t,L) dt + \int_{T_1}^{\tau} \int_0^L \phi yy_x dxdt  \notag\\
 & + \int_0^L y(\tau,x)\phi(\tau,x) dx - \int_0^L y_0 \phi(s,x) dx = 0.
\end{align}
This shows that $y$ is a solution to \eqref{likdv}, with $T_1:=s$, $T_2$ arbitrary in $(s,+\infty)$, $\tilde{h}:=-yy_x\in L^1([s,T_2]; L^2(0,L))$ and
$h=u(\cdot,y(\cdot))\in L^2(s,T_2)$. Let us emphasize that, by Lemma~\ref{lem1}, it also implies that $y\in \mathcal{B}_{s,T} $ for every
 $T\in (s,+\infty)$. This concludes
the proof of \eqref{eq-y-solution} and  of  Theorem~\ref{th-feedback-not-Lipschitz}.
\end{proof}

\section{Proof of Proposition \ref{prop-order-2-H1-au-lieu-de-L2}}
\label{sec-appendix-C}
Let us first recall that Proposition \ref{prop-order-2-H1-au-lieu-de-L2} is due to Eduardo Cerpa if one requires only $u$ to be in $L^2(0,T)$ instead of being in $H^1(0,T)$ : see \cite[Proposition 3.1]{cerpa07} and \cite[Proposition 3.1]{cerpa09}.
In his proof, Eduardo Cerpa  uses Lemma \ref{lem4}, the controllability in $H$ with controls $u \in L^2$. Actually, the only place in Eduardo Cerpa's proof where the controllability in $H$ is used is in page 887 of \cite{cerpa07} for  the construction of $\alpha_1$, where, with the notations of \cite{cerpa07} $\Re(y_{\lambda})$, $\Im(y_{\lambda}) \in H$. We notice that $\Re(y_{\lambda})$, $\Im(y_{\lambda})$ share more regularity and better boundary conditions. Indeed, one has
 \begin{gather*}\label{ylam}
 \begin{cases}
 \lambda y_{\lambda} + y_{\lambda}^{'} +  y_{\lambda}^{'''} = 0,\\
 y_{\lambda}(0) = y_{\lambda}(L) = 0,
 \end{cases}
 \end{gather*}
which implies that
 \[  \Re (y_{\lambda}), \Im (y_{\lambda})  \in \mathcal{H}^3,            \]
where
 \begin{equation}
 \mathcal{H}^3 := H \cap \{  \omega \in H^3(0,L); \omega(0) = \omega(L) = 0   \}.
\end{equation}
In order to adapt Eduardo Cerpa's proof in the framework of $u \in H^1(0,T)$, it is  sufficient to prove  the following controllability result in $\mathcal{H}^3$ with control $u \in H^1(0,T)$.
 \begin{proposition}\label{prop-controllability-with-H1-control}
For every $y_0$, $y_1 \in \mathcal{H}^3$ and for every $T >0$,
there exists a control $u \in H^1(0,T)$ such that the solution $y \in \mathcal{B}$ to the Cauchy problem
\begin{gather*}
\begin{cases}
y_t+y_{xxx}+y_x = 0,\\
y(t,0)= y(t,L)=0, \\
y_x(t,L)= u(t),
\\
y(0,\cdot) = y_0,
\end{cases}
\end{gather*}
satisfies $y(T,\cdot) = y_1$.
 \end{proposition}

 The proof of Proposition \ref{prop-order-2-H1-au-lieu-de-L2} is the same as the one of \cite[Proposition 3.1]{cerpa07}, with the only difference  that one uses  Proposition \ref{prop-controllability-with-H1-control} instead of Lemma \ref{lem4}.

\begin{proof}[Proof of Proposition~\ref{prop-controllability-with-H1-control}]
Let us first point out that $0$ is not an eigenvalue of the operator $\mathcal{A}$. Indeed this follows from Property \hyperref[P2]{($\mathcal{P}_2$)}, \eqref{varj} and \eqref{eq-omega}. Using Lemma~\ref{lem4} and
\cite[Proposition 10.3.4]{2009-Tucsnak-Weiss-book} with $\beta=0$, it suffices to check that

\begin{equation}\label{solvability}
\text{for every $f \in H$, there exists $y\in \mathcal{H}^3$ such that $-y_{xxx}-y_x=f$}.
\end{equation}
Let $f\in H$. We know that there exists $y\in H^3(0,L)$ such that
\begin{gather}
\label{equation-y-reg}
-y_{xxx}-y_x=f,
\\
\label{condition-bord-y-reg}
y(0)=y(L)=y_x(L)=0.
\end{gather}
Simple integrations by parts, together with \eqref{varphi}, \eqref{varphi2}, \eqref{equation-y-reg} and \eqref{condition-bord-y-reg},  show that, with $\varphi:=\varphi_1+i\varphi_2$,
\begin{equation}\label{scal-0}
0=\int_0^Lf\varphi  dx = \int_0^L(-y_{xxx}-y_x)\varphi dx =\int_0^Ly(\varphi_{xxx}+\varphi_{x})dx =
i\frac{2\pi}{p}\int_0^Ly\varphi dx,
\end{equation}
which, together with \eqref{condition-bord-y-reg}, implies that $y\in \mathcal{H}^3$. This concludes the proof of
  \eqref{solvability} as well as the proof of Proposition~\ref{prop-controllability-with-H1-control}
  and of Proposition~\ref{prop-order-2-H1-au-lieu-de-L2}.

\end{proof}

\end{document}